\documentclass[10pt,reqno]{amsart}
\usepackage{latexsym,amsmath,amssymb,amscd}
\usepackage[all]{xy}
\usepackage{enumerate}
\usepackage{hyperref}

\def\today{\ifcase \month \or
	January \or February \or March \or April \or
	May \or June \or July \or August \or
	September \or October \or November \or December \fi
	\space\number\day , \number\year}

\begin{document}
	

	\newcommand{\bfi}{\bfseries\itshape}
	
	\newtheorem{theorem}{Theorem}
	\newtheorem{acknowledgment}[theorem]{Acknowledgment}
	\newtheorem{corollary}[theorem]{Corollary}
	\newtheorem{definition}[theorem]{Definition}
	\newtheorem{example}[theorem]{Example}
	\newtheorem{lemma}[theorem]{Lemma}
	\newtheorem{notation}[theorem]{Notation}
	\newtheorem{problem}[theorem]{Problem}
	\newtheorem{proposition}[theorem]{Proposition}
	\newtheorem{question}[theorem]{Question}
	\newtheorem{remark}[theorem]{Remark}
	\newtheorem{setting}[theorem]{Setting}
	
	\numberwithin{theorem}{section}
	\numberwithin{equation}{section}

\renewcommand{\1}{{\bf 1}}
\newcommand{\Ad}{{\rm Ad}}
\newcommand{\Aut}{{\rm Aut}\,}
\newcommand{\ad}{{\rm ad}}
\newcommand{\botimes}{\bar{\otimes}}
\newcommand{\Ci}{{\mathcal C}^\infty}
\newcommand{\Cl}{{\rm Cl}\,}
\newcommand{\card}{{\rm card}}
\newcommand{\clgth}{\text{\rm c-length}\,}
\newcommand{\cortex}{{\rm Cor}}
\newcommand{\Der}{{\rm Der}\,}
\newcommand{\de}{{\rm d}}
\newcommand{\dr}{{\rm dr}\,}
\newcommand{\ee}{{\rm e}}
\newcommand{\End}{{\rm End}\,}
\newcommand{\flgth}{\text{\rm f-length}\,}
\newcommand{\id}{{\rm id}}
\newcommand{\ie}{{\rm i}}
\newcommand{\Index}{{\rm Index}}
\newcommand{\GL}{{\rm GL}}
\newcommand{\Gr}{{\rm Gr}}
\newcommand{\Hom}{{\rm Hom}\,}
\newcommand{\Ind}{{\rm Ind}}
\newcommand{\ind}{{\rm ind}\,}
\newcommand{\Ker}{{\rm Ker}\,}
\newcommand{\lgth}{{\rm length}\,}
\newcommand{\Prim}{{\rm Prim}}
\newcommand{\pr}{{\rm pr}}
\newcommand{\Ran}{{\rm Ran}\,}
\newcommand{\RRa}{{\rm RR}}
\newcommand{\rank}{{\rm rank}\,}
\renewcommand{\Re}{{\rm Re}\,}
\newcommand{\SO}{{\rm SO}\,}
\newcommand{\sa}{{\rm sa}}
\newcommand{\spa}{{\rm span}\,}
\newcommand{\tsr}{{\rm tsr}}
\newcommand{\Tr}{{\rm Tr}\,}

\newcommand{\CC}{{\mathbb C}}
\newcommand{\HH}{{\mathbb H}}
\newcommand{\RR}{{\mathbb R}}
\newcommand{\TT}{{\mathbb T}}

\newcommand{\Ac}{{\mathcal A}}
\newcommand{\Bc}{{\mathcal B}}
\newcommand{\Cc}{{\mathcal C}}
\newcommand{\Ec}{{\mathcal E}}
\newcommand{\Fc}{{\mathcal F}}
\newcommand{\Gc}{{\mathcal G}}
\newcommand{\Hc}{{\mathcal H}}
\newcommand{\Ic}{{\mathcal I}}
\newcommand{\Jc}{{\mathcal J}}
\newcommand{\Kc}{{\mathcal K}}
\newcommand{\Lc}{{\mathcal L}}
\renewcommand{\Mc}{{\mathcal M}}
\newcommand{\Nc}{{\mathcal N}}
\newcommand{\Oc}{{\mathcal O}}
\newcommand{\Pc}{{\mathcal P}}
\newcommand{\Sc}{{\mathcal S}}
\newcommand{\Tc}{{\mathcal T}}
\newcommand{\Uc}{{\mathcal U}}
\newcommand{\Vc}{{\mathcal V}}
\newcommand{\Xc}{{\mathcal X}}
\newcommand{\Yc}{{\mathcal Y}}
\newcommand{\Wc}{{\mathcal W}}
\newcommand{\Zc}{{\mathcal Z}}

\newcommand{\ag}{{\mathfrak a}}
\newcommand{\fg}{{\mathfrak f}}
\renewcommand{\gg}{{\mathfrak g}}
\newcommand{\hg}{{\mathfrak h}}
\newcommand{\kg}{{\mathfrak k}}
\newcommand{\mg}{{\mathfrak m}}
\renewcommand{\ng}{{\mathfrak n}}
\newcommand{\pg}{{\mathfrak p}}
\newcommand{\sg}{{\mathfrak s}}
\newcommand{\zg}{{\mathfrak z}}

\newcommand{\ZZ}{\mathbb Z}
\newcommand{\NN}{\mathbb N}

\title[The isomorphism problem for $C^*$-algebras of Lie groups]{On the isomorphism problem for $C^*$-algebras of nilpotent Lie groups}
\author{Ingrid Belti\c t\u a and Daniel Belti\c t\u a}
\address{Institute of Mathematics ``Simion Stoilow'' 
	of the Romanian Academy, 
	P.O. Box 1-764, Bucharest, Romania}
\email{ingrid.beltita@gmail.com, Ingrid.Beltita@imar.ro}
\email{beltita@gmail.com, Daniel.Beltita@imar.ro}

\thanks{This research was partly  supported by Project MTM2013-42105-P, fondos FEDER, Spain.}
\date{\today}

\begin{abstract} 
We investigate    
to what extent a nilpotent Lie  group is determined by its $C^*$-algebra. 
We prove that, within the class of exponential Lie groups, direct products of Heisenberg groups with abelian Lie groups are uniquely determined even by their unitary dual, while nilpotent Lie groups of dimension $\le 5$ are uniquely determined by the Morita equivalence class of their $C^*$-algebras.
We also find that this last property is shared by the filiform Lie groups and by the $6$-dimensional   free two-step nilpotent Lie group. 

\end{abstract}

\maketitle

\section{Introduction}

The main results of the present paper are related to 
the following inverse problem in noncommutative harmonic analysis:  
To what extent a locally compact group is determined by its representation theory?  
If $G$ is a locally compact group, then important information on its representation theory is encoded by the natural topology of its unitary dual $\widehat{G}$, that is, the set of all equivalence classes of unitary irreducible representations of $G$. 
When  $G$ is type $I$, additional information is encoded by the group $C^*$-algebra of~$G$, 
denoted by~$C^*(G)$, whose space of primitive ideals is canonically homeomorphic to the dual space $\widehat{C^*(G)}$, which is further homeomorphic to~$\widehat{G}$. 
The \emph{isomorphism problem} referred to in the title may be stated as follows: If $G_1$ and $G_2$ are locally compacr groups whose $C^*$-algebras are isomorphic, are $G_1$ and $G_2$ necessarily isomorphic as locally compact  groups?

Locally compact groups cannot, in general,  be recovered from their $C^*$-algebras.
For instance, for all compact infinite Lie groups (or even compact infinite metrizable groups), their unitary dual spaces are countably infinite discrete topological spaces, 
hence are homeomorphic to each other. 
As commutative $C^*$-algebras having homeomorphic spaces of primitive ideals are $*$-isomorphic, 
it then follows that, in particular,  
\emph{the $C^*$-algebras of all compact abelian Lie groups of dimension $\ge1$ are mutually $*$-isomorphic}, hence for this class of Lie groups even the dimension of a group cannot be read off from its $C^*$-algebra. 
This phenomenon is not confined to compact groups, as several examples of non-isomorphic exponential Lie groups having isomorphic $C^*$-algebras were pointed out in \cite{LiLu13} and \cite{BB17b}. 

Nevertheless, it is rather easy to see that abelian simply connected Lie groups are uniquely determined by their $C^*$-algebras, within the class of the connected simply connected abelian Lie groups. 
Moreover, as we show below (Lemmma~\ref{abelian})
the unitary dual   distinguishes these abelian  groups within the class of the connected simply connected exponential Lie groups.

A natural question, that we  address in this paper, can be then stated as follows:
If $G_1$ and $G_2$ are 
\textit{nilpotent Lie groups} whose $C^*$-algebras are isomorphic, are $G_1$ and $G_2$ necessarily isomorphic as Lie groups? 
This open problem  was also stated in \cite{LiLu13}.  
Since we want to determine groups by their unitary representation theory, it is more natural to study groups that have Morita equivalent $C^*$-algebras, not necessarily isomorphic. As noted  in \cite[pages 160--161]{Ro94}: 
\begin{quotation}
	``Two $C^*$-algebras are called strongly Morita equivalent if, roughly speaking, their $*$-representation theories are identical.
		Since the purpose of introducing $C^*(G)$ is 
		to understand the unitary representation theory of $G$, which is equivalent
		to the $*$-representation theory of $C^*(G)$, the natural goal in studying unitary
		representations of a group $G$ from Rieffel's point of view is thus to classify $C^*(G)$
	up to Morita equivalence.''
\end{quotation}

It is convenient to make the following definition. 

\begin{definition}
	\normalfont
	An exponential Lie group $G$ is called \emph{stably $C^*$-rigid} if for every other  exponential Lie group $H$ one has 
	$$G\simeq H\iff C^*(G), C^*(H) \text{ are Morita equivalent}.$$
\end{definition}

Thus the problem stated above can be replaced by a more general one: 
Are nilpotent Lie groups stably $C^*$-rigid?
Indeed, in this paper we indicate examples of noncommutative nilpotent Lie groups that are stably $C^*$-rigid. 
To this end we use a range of tools including 
notions of multiplicity of limit points (see \cite{Lu90}, \cite{ArSp96}, \cite{ArSoSp97}, \cite{AKLSS01}, and \cite{AaH12}) and 
of topological dimension (see \cite{BB16b} and \cite{BB17}) that are suitable for the study of noncommutative Lie groups and their $C^*$-algebras. 
We note that $C^*$-rigidity properties of discrete nilpotent groups 
 were recently obtained by completely different methods (see for instance \cite{ER18}).

\subsection*{Main results and structure of this paper}
The paper starts with Section~\ref{prelim}, which contains the needed preliminaries,  and also technical results 
on topological spaces. 
In particular, for a nilpotent Lie group $G$, we introduce the quantity  $\ind{G}$  that relates the 
dimension $k$ such that there is an open dense subset of $\hat G$ that is homeomorphic to an open  subset of $\RR^k$ and 
the maximal dimension of a coadjoint orbit of $G$.  
This quantity will turn to be important in the next sections. 
Also, we obtain some useful results on the maximal closed and Hausdorff subsets of locally quasi-compact 
topological spaces. 

Section~\ref{Sect3} contains  a detailed study
of set of characters of a nilpotent Lie group, as maximal closed and Hausdorff subsets of the unitary spectrum, 
and links the properties of this set with those of the real rank of the group $C^*$-algebra.

The first result on $C^*$-rigidity of the paper proves that, 
within the class of exponential  Lie groups, the groups that are direct product 
of a Heisenberg group with a vector space group are uniquely determined by their spectrum, 
which is  actually a stronger property than stably $C^*$-rigidity.

\begin{theorem}\label{2H}
	Let $G_1$ a nilpotent Lie group with Lie algebra $\gg_1$ and assume that
	$\dim [\gg_1, \gg_1]\le 1$. 
	Then
	for any exponential Lie group $G_2$, 
	$\widehat{G}_1$ is homeomorphic  
	to $\widehat{G}_2$ if and only if $G_1$ is isomorphic to $G_2$.
\end{theorem}
The proof of Theorem~\ref{2H} is given in Section~\ref{proof-th}, and it is based on the characterization of the  set of characters of nilpotent Lie groups
performed in Section~\ref{Sect3}. 
A weaker result of this type for Heisenberg groups was established in \cite{BBL17}, 
which used however some structures that cannot be encoded by the group $C^*$-algebras. 
Using a completely different method we obtain this result on its natural level of generality in Theorem~\ref{2H}.

In Section~\ref{proof-cor} we analyze the nilpotent Lie groups of dimension $\le 5$. 
In general, in this case we need more than only the unitary dual to distinguish a group, namely we need to pinpoint a special open dense subset given by the spectrum of the largest bounded trace ideal in the $C^*$-algebra of the group.
 The result in this case is the following theorem.
\begin{theorem}\label{5D}
	Any nilpotent Lie group of dimension $\le 5$ is stably  $C^*$-rigid.
	\end{theorem}
Theorem~\ref{5D} is proved in the last subsection of Section~\ref{proof-cor}, 
To prove the theorem we first study a class of more tractable Lie groups of arbitrary dimension, namely the groups whose all nontrivial coadjoint orbits have the same dimension, and prove that this class of is invariant, within the class of nilpotent Lie groups,  under the Morita equivalence of their $C^*$-algebras (Proposition~\ref{P4}). Some further useful properties of these tractable groups are given in the first subsection of Section~\ref{proof-cor}.
Theorem~\ref{2H} is also used in the proof of Theorem~\ref{5D}.

Similar rigidity results are established Section~\ref{other_ex}  for the filiform (threadlike) Lie groups of arbitrary dimension (Theorem~\ref{fili_th})  and for the free two-step Lie group of dimension~$6$
(Theorem~\ref{N6-15}). 
The unitary dual of filiform Lie groups has been used to illustrate the intricacy of the group $C^*$-algebras in \cite{AKLSS01},  \cite{ArSoKaSc99}, \cite{Lu90}, and the references therein.

\section{Preliminaries and technical results}\label{prelim}

\subsection{General notation and terminology} 
We denote the Lie groups by upper case Roman letters and their corresponding Lie algebras by the corresponding lower case Gothic letters. 
By nilpotent Lie group we always mean a connected simply connected nilpotent Lie group. 
By exponential Lie group we mean a Lie group $G$ whose exponential map $\exp_G\colon\gg\to G$ is bijective. 

For an integer $k \ge 1$, we denote by $\ag_k$ the commutative nilpotent Lie algebra of dimension $k$, which is nothing else than a real vector space
of dimension $k$. Then $A_k$ is the corresponding nilpotent Lie group. 
Also, we denote by $H_{2k+1}$  the Heisenberg group of dimension $2k+1$, and by  $\hg_{2k+1}$  its Lie algebra.

For a Lie group $G$ with Lie algebra $\gg$, and Lie bracket 
$[\cdot,\cdot ]$,  we denote by $\gg^*$ the dual of $\gg$ and by $\langle \cdot, \cdot\rangle \colon \gg^* \times \gg \to \RR$ the corresponding duality
pairing. 
For any subset $\sg \subseteq \gg$, we set 
$$ \sg^\perp :=\{\xi \in \gg^* \mid \langle {\xi}, {\sg} \rangle =0\}.$$
We denote its corresponding coadjoint action by 
$\Ad^*_G\colon G\times \gg^*\to \gg^*$. 
The space of all coadjoint orbits is denoted by $\gg^*/G$, seen as a topological space with its natural quotient topology.
The corresponding quotient map is denoted by 
$$q\colon \gg^* \to \gg^*/G, \quad \xi \mapsto \Oc_\xi, $$ 
hence for every $\xi\in \gg^*$  
its corresponding coadjoint orbit is 
$$\Oc_\xi: =\Ad^*_G(G)\xi.$$ 
The coadjoint isotropy subalgebra at $\xi \in \gg^*$ is 
$$\gg(\xi):=\{X\in \gg\mid \langle{\xi}, {[X, \gg]}\rangle=0\}.$$
We recall that for  any locally compact group $G$,  $\widehat{G}$ stands for the set of unitary equivalence classes of unitary irreducible representations of~$G$ with 
the natural topology 
defined in terms of the group $C^*$-algebra $C^*(G)$. 
Then, in the case of exponential Lie groups, we  may tacitly identify 
$$ \gg^*/G \simeq \widehat G \simeq\widehat{ C^*(G)}, $$
via canonical homeomorphism (see  
\cite{CG90},  \cite{FuLu15},  and \cite{Dix77}).

Two separable $C^*$-algebras $\Ac_1$ and $\Ac_2$ are Morita equivalent if and only if they are stably isomorphic, in the sense that there is a  $*$-isomorphism $\Ac_1\otimes\Kc\simeq\Ac_2\otimes\Kc$, where $\Kc$ is the $C^*$-algebra of compact operators on a separable infinite-dimensional complex Hilbert space.   
If two $C^*$-algebras $\Ac_1$ and $\Ac_2$ are Morita equivalent, then there is a homeomorphism $\widehat{\Ac_1} \simeq \widehat{\Ac_2}$ (see \cite{RaWi98}).

We use  the notion of real rank $\RRa(\Ac)$ for a $C^*$-algebra $\Ac$ and results on the real rank for $C^*$-algebras 
of exponential Lie groups. See \cite{BB16b} and the references therein.

Finally, we denote by $\RR$ and $\CC$ the fields of real and complex numbers, respectively. 
We also denote $\RR^\times:=\RR\setminus\{0\}$ and $\TT:=\{z\in\CC\mid \vert z\vert=1\}$, and both these sets are usually regarded as 1-dimensional Lie groups with respect to the group operation given by multiplication.

	\subsection{Special $\RR$-spaces}
	
	The notions introduced in Definitions \ref{R-space} and \ref{solvspecl_space} were suggested by the special topological properties of unitary dual spaces of nilpotent Lie groups established in \cite{BBL17}. 
	
	\begin{definition}\label{R-space}
		\normalfont
		A \emph{special $\RR$-space} is a topological space $X$ endowed with a 
		continuous map $\RR\times X\to X$, $(t,x)\mapsto t\cdot x$, called \emph{structural map},  and with a \emph{distinguished point} $x_0\in X$ 
		satisfying the following conditions: 
		\begin{enumerate}
			\item For every $x\in X$ and $t\in \RR$ one has $0\cdot x= t \cdot x_0=x_0$ and $1\cdot x=x$. 
			\item For all $t,s\in\RR$ and $x\in X$ one has $t\cdot(s\cdot x)=ts\cdot x$. 
			\item\label{R-space_item3}
			For every $x\in X\setminus \{x_0\}$ the map $\psi_x\colon \RR\to X$, $t\mapsto t\cdot x$ is a homeomorphism onto its image. 
		\end{enumerate}
		An \emph{$\RR$-subspace} of the  special $\RR$-space $X$ is any subset $\Gamma\subseteq X$ 
		such that $\RR^\times \cdot \Gamma\subseteq\Gamma$. 
		If this is the case, then $\Gamma\cup\{x_0\}$ is a special $\RR$-space on its own. 
		
		If $Y$ is another special $\RR$-space with its structural map $\RR\times X\to X$, $(t,x)\mapsto t\cdot x$ and its distinguished point $y_0\in Y$, then a map $\psi\colon X\to Y$ is called an \emph{isomorphism of special $\RR$-spaces} if $\psi$ is a homeomorphism and $\psi(t\cdot x)=t\cdot \psi(x)$ for all $t\in\RR$ and $x\in X$.

		In the  notation above, a function $\varphi\colon X\to\RR$ is called \emph{homogeneous} if there exists $r\in[0,\infty)$ 
		such that $\varphi(t\cdot x)=t^r\varphi(x)$ for all $t\in\RR$ and $x\in X$. 
	\end{definition} 
	
	\begin{remark}\label{R-space_rem}
		\normalfont
		If $X$ is a special $\RR$-space, then one has: 
		\begin{enumerate}[(i)]
			\item If $t\in\RR^\times$ and $x\in X\setminus\{x_0\}$, then $t\cdot x=x_0$ if and only if $t=0$. 
			\item If $\psi\colon X\to Y$ is  an isomorphism of special $\RR$-spaces then $\psi(x_0)=y_0$. 
		\end{enumerate}
	\end{remark}
	
	\begin{example}\label{complem}
		\normalfont
		If $X$ is a special $\RR$-space and $\Gamma\subseteq X$ is an $\RR$-subspace of $X$, then $X\setminus\Gamma$ is also an $\RR$-subspace of~$X$.
	\end{example}
	
	\begin{example}
		\label{1dim}
		\normalfont
		If $X$ is a special $\RR$-space and $X$ is homeomorphic to $\RR$, then one can check that for every $x\in X\setminus\{x_0\}$ the map 
		$\psi_x\colon\RR\to X$, $t\mapsto t\cdot x$, is a homeomorphism, and moreover $\psi_x$ is an isomorphism of $\RR$-spaces between $X$ and the $\RR$-space $\RR$ regarded as a 1-dimensional real vector space. 
	\end{example}

	\begin{definition}\label{solvspecl_space}
		\normalfont
		A topological space $X$ is a
		 \emph{$\RR$-space of finite length} 
		if it is second countable, is locally quasi-compact, has the property $T_1$, and 
		there exists an increasing finite family of open subsets,  
		\begin{equation}\label{solvseries}
		\emptyset=V_0\subseteq V_1\subseteq\cdots\subseteq V_n=X, 
		\end{equation}
		such that $V_j\setminus V_{j-1}$ is Hausdorff in its relative topology and  dense in  $X\setminus V_{j-1}$, for $j=1,\dots,n$,  and satisfy 
		the following additional properties: 
		\begin{enumerate}
			\item\label{solvspecl_space_item1} 
			$X$ has the structure of a special $\RR$-space 
			and $\Gamma_j:=V_j\setminus V_{j-1}\subseteq X$ is an $\RR$-subspace for $j=1,\dots,n$. 
			\item\label{solvspecl_space_item2} 
			$\Gamma_n:=X\setminus V_{n-1}$ is isomorphic as a special $\RR$-space to a finite-dim\-ensional vector space, whose origin is the distinguished point $x_0$ of $X$. 
			\item\label{solvspecl_space_item3} 
			For $j=1,\dots,n-1$ the points of $\Gamma_{j+1}$ are closed and separated in $X\setminus V_{j}$.  
			\item\label{solvspecl_space_item4} 
			For $j=1,\dots,n$, $\Gamma_j$ is isomorphic as a special $\RR$-space to a 
			cone $C_j$
			in a finite-dimensional vector space. 
			In addition, $C_1$ is assumed to be a semi-algebraic Zariski open set,  
			and the dimension of the corresponding ambient vector space is called the \emph{index of $X$} and is denoted by $\ind X$. 
			\end{enumerate}
 If $G$ is a nilpotent Lie group with Lie algebra $\gg$,  then its unitary dual $\widehat{G}$ is a $\RR$-space of finite length by \cite[Thm. 4.11]{BBL17}, 
		and we define $\ind G=\ind \gg :=\ind\widehat{G}$.
		\end{definition}
	More details on the above definition could be found in \cite{BB17}.

	\begin{lemma}\label{27december2017-1555}
		If $G$ is a nilpotent Lie group and $r\in\NN$, then the following conditions are equivalent: 
		\begin{enumerate}[{\rm(i)}]
			\item $\ind G=r$. 
			\item There exists an open dense subset $V\subseteq\widehat{G}$ such that $V$ is homeomorphic to an open subset of~$\RR^r$. 
			\item $r=\min\{\dim\gg(\xi)\mid \xi \in \gg^*\}$.
		\end{enumerate}
		If $r=1$, then the above conditions are further equivalent to the following one: 
		\begin{enumerate}[{\rm(i)}]
			\setcounter{enumi}{2}
			\item The center $Z$ of $G$ is 1-dimensional and there exists $\xi\in\gg^*$ whose coadjoint isotropy group is equal to $Z$ and  the mapping $\RR^\times\to\gg^*/G$, $t\mapsto\Oc_{t\xi}$, is a homeomorphism of $\RR^\times$ onto an open dense subset of $\gg^*/G$. 
		\end{enumerate}
	\end{lemma}
	
	\begin{proof}
		Use \cite[Prop. 4.9]{BBL17}.	
	\end{proof}
	
	\begin{remark}\label{27december2017-1600}
		\normalfont
		If $G_1$ and $G_2$ are nilpotent Lie groups 
		such that $\widehat{G}_1$ is homeomorphic to $\widehat{G}_2$, 
		then $\ind G_1=\ind G_2$, as it follows easily from
		Lemma~\ref{27december2017-1555}.
	\end{remark}

	\subsection{Maximal closed and Hausdorff subsets of a locally quasi-compact topological space}
	The results in this subsection will be used in  Section~\ref{proof-cor}.
	We start with a notation. 
\begin{notation}
	\normalfont
	
	a)	For a locally quasi-compact space $X$ 	we denote by  by $Q(X)$ the set of all closed connected subsets $S\subseteq X$ for which the relative topology of $S$ is Hausdorff, and endow $Q(X)$ with the partial ordering given by inclusion. 
	
	For any net $\{x_j\}_{j\in J}$ in $X$ we also denote by 
	$\liminf\limits_{j\in J}\{x_j\}$ the set of its limit points, 
	that is, the points $x\in X$ having the property that for every neighborhood $V$ of $x$ there exists $j_V\in J$ for which $x_j\in V$ for all $j\in J$ with $j\ge j_V$. 
	
	b) For a nilpotent  Lie group $G$ we have identified
	$\gg^*/G\simeq \widehat G$ and thus we can use the notation $Q(\widehat G): = Q(\gg^*/G)$.

\end{notation}

\begin{lemma}\label{L5a}
	Let $X$ be a locally quasi-compact space with an open subset $D\subseteq X$ satisfying the following conditions:  
	\begin{enumerate}[{\rm(a)}]
		\item $D\subsetneqq\overline{D}$,  and the relative topology of $D$ is locally connected and Hausdorff.  
		\item For every net $\{x_j\}_{j\in J}$ contained in $D$ 
		with $(X\setminus D)\cap\liminf\limits_{j\in J}\{x_j\}\ne\emptyset$, 
		one has $\vert \liminf\limits_{j\in J}\{x_j\}\vert\ge 2$. 
	\end{enumerate}
	Then for every maximal element $S$ of $Q(X)$ one has $S\subseteq X\setminus D$. 
\end{lemma}

\begin{proof}
	We may assume $S\ne\emptyset$ and we argue by contradiction: suppose $S\cap D\ne\emptyset$. 
	Then, since $S$ is connected, either $S\cap D$ is not relatively closed in $S$ or $S\subseteq D$. 
	We discuss these cases separately. 
	
	Case 1: $S\cap D$ is not relatively closed in $S$. 
	Then there exists a net $\{x_j\}_{j\in J}$ in $S\cap D$ with 
	$(S\setminus D)\cap\liminf\limits_{j\in J}\{x_j\}\ne\emptyset$. 
	Then the hypothesis implies 
	$\vert \liminf\limits_{j\in J}\{x_j\}\vert\ge 2$. 
	On the other hand, since $S$ is closed, one has $\liminf\limits_{j\in J}\{x_j\}\subseteq S$, 
	hence we obtain a contradiction with the fact that $S$ is Hausdorff with respect to its relative topology. 
	
	Case 2: $S\subseteq D$. 
	Since $S$ is closed and $D\ne\overline{D}$, 
	one then actually has $S\subsetneqq D$. 
	Since $S\ne\emptyset$, there exists $s_0\in S\setminus{\rm int}\, S$. 
	In particular $s_0\in D$ and then, since $D$ is locally compact, 
	$s_0$ has a fundamental system of closed neighborhoods contained 
	in~$D$. 
	Therefore, as $D$ is also locally connected, there exists a closed connected neighborhood $S_0$ of $s_0$ with $S_0\subseteq D$. 
	
	One has $s_0\in S_0\cap S$ hence, since both $S_0$ and $S$ are connected, it follows that $S_0\cup S$ is connected. 
	Moreover, $S_0\cup S\subseteq D$, hence the relative topology of $S_0\cup S$ is Hausdorff. 
	Also, both $S_0$ and $S$ are closed, hence  $S_0\cup S$ is closed, 
	and then $S_0\cup S\in Q(X)$. 
	
	On the other hand, since $s_0\in S\setminus{\rm int}\, S$ and $S_0$ is a neighborhood of $s_0$, 
	one cannot have $S_0\subseteq S$, hence $S\subsetneqq S_0\cup S$, 
	and this is a contradiction with the hypothesis that $S$ is a maximal element of $Q(X)$. 
	This completes the proof. 
\end{proof}

\begin{lemma}\label{L5b}
	Let $X$ be a locally quasi-compact space with a finite family of open subsets 
	$$\emptyset=D_0\subseteq D_1\subseteq\cdots\subseteq D_n= X$$ 
	satisfying the following conditions for $k=1,\dots,n-1$: 
	\begin{enumerate}[{\rm(a)}]
		\item 
		$D_k\setminus D_{k-1}$ is not a closed subset of $X$ and its 
		relative topology is locally connected and Hausdorff.  
		\item For every net $\{x_j\}_{j\in J}$ contained in $D_k\setminus D_{k-1}$ 
		with $(X\setminus D_k)\cap\liminf\limits_{j\in J}\{x_j\}\ne\emptyset$, 
		one has $\vert \liminf\limits_{j\in J}\{x_j\}\vert\ge 2$. 
	\end{enumerate}
	Then for every maximal element $S$ of $Q(X)$ one has $S\subseteq X\setminus D_{n-1}$. 
\end{lemma}

\begin{proof}
	We argue by induction. 
	The case $n=1$ is trivial, and the case $n=2$ is exactly Lemma~\ref{L5a}. 
	
	We now assume $n\ge 3$ and that the assertion holds for $n-1$. 
	Using the induction hypothesis for the family 
	$$\emptyset=D_0\subseteq D_1\subseteq\cdots\subseteq D_{n-2}\subseteq X$$ 
	we obtain 
	\begin{equation}\label{L5b_proof_eq1}
	S\subseteq X\setminus D_{n-2}.
	\end{equation}
	We now check that Lemma~\ref{L5a} can be applied to  $$\widetilde{D}:=D_{n-1}\setminus D_{n-2}\subseteq X\setminus D_{n-2}=:\widetilde{X}.$$
	It directly follows by hypothesis that $\widetilde{X}$ is locally quasi-compact, and $\widetilde{D}$ is an open subset of $\widetilde{X}$ whose relative topology is locally connected and Hausdorff.  
	To show that $\widetilde{D}$ is not closed in $\widetilde{X}$, we argue by contradiction: suppose that 
	$\widetilde{D}$ is a closed subset of $\widetilde{X}$.  
	Then, since $\widetilde{X}$ is a closed subset of $X$, it follows that $\widetilde{D}=D_{n-1}\setminus D_{n-2}$ is closed in $X$, 
	which is a contradiction with the hypothesis. 
	
	Now let $\{x_j\}_{j\in J}$ be a net in $\widetilde{D}$ with  
	$(\widetilde{X}\setminus \widetilde{D})\cap\liminf\limits_{j\in J}\{x_j\}\ne\emptyset$. 
	Since 
	$\widetilde{X}\setminus \widetilde{D}=X\setminus D_{n-1}$, 
	it then follows by hypothesis that 
	$\vert \liminf\limits_{j\in J}\{x_j\}\vert\ge 2$.
	
	Furthermore, since $\widetilde{X}$ is a closed subset of $X$, it follows that $Q(\widetilde{X})\subseteq Q(X)$ and $S\in Q(\widetilde{X})$ by \eqref{L5b_proof_eq1}, hence $S$ is a maximal element of $Q(\widetilde{X})$ as well. 
	It then follows by Lemma~\ref{L5a} that $S\subseteq \widetilde{X}\setminus\widetilde{D}=X\setminus D_{n-1}$, 
	and this completes the proof by induction. 
\end{proof}
\subsection{The space of closed subgroups of a locally compact group}

For any locally compact group~$G$  we consider 
$$\Kc(G):=\{K\mid K\text{ closed subgroup}\subseteq G\}.$$ 
It is well known that $\Kc(G)$ is a compact topological space with its Fell topology, for which a basis   consists
of the sets 
\begin{equation}
\label{Fellbasis}
\Uc(C,\Sc):=\{K\in\Kc(G)\mid K\cap C=\emptyset;\ (\forall A\in\Sc)\ K\cap A\ne\emptyset\}
\end{equation}
for all compact sets $C\subseteq G$ and all finite sets $\Sc$ of open subsets of~$G$. 
 
The following simple fact is probably well known. 

\begin{lemma}\label{points}
	Let $G$ be a first countable locally compact group.  
	If $(K_j)_j$ and $K$  are closed subgroups of $G$ with  $\lim\limits_{j\to\infty}K_j=K$ in $\Kc(G)$, 
	then for every finite family $x^{(1)},\dots,x^{(m)}\in K$ 
	there exist 
	a subsequence $\{K_{j_i}\}_{i\ge 1}$ and 
	$x^{(1)}_i,\dots,x^{(m)}_i\in K_{j_i}$ for every $i\ge 1$ with $\lim\limits_{i\to\infty}x^{(s)}_i=x^{(s)}$ in $G$ 
	for $s=1,\dots,m$. 
\end{lemma}

\begin{proof} 
	We first assume $m=1$ and denote $x:=x^{(1)}$. 
	Let $A_1\supseteq A_2\supseteq\cdots$ be a countable base of open neighborhoods of $x$ in $G$. 
	For every $i\ge 1$ one has  
	$$\Uc(\emptyset,\{A_i\}):=\{H\in\Kc(G)\mid  H\cap A_i\ne\emptyset\} $$
	which is  an open subset of $\Kc(G)$, 
	cf.~\eqref{Fellbasis}.  
	Since $x\in K\cap A_i$, one has $K\in \Uc(\emptyset,\{A_i\})$, 
	hence there exists $j_i\ge 1$ with $K_j\in \Uc(\emptyset,\{A_i\})$, that is, $K_j\cap A_i\ne\emptyset$, 
	for every $j\ge j_i$, 
	by the hypothesis $\lim\limits_{j\to\infty}K_j=K$ in $\Kc(G)$. 
	We may  assume $j_1<j_2<\cdots$. 
	Selecting any $x_i\in K_{j_i}\cap A_i$ and using the fact that  
	$A_1\supseteq A_2\supseteq\cdots$ is a base of neighborhoods of $x$, 
	one can check that $\lim\limits_{i\to\infty}x_i=x$ in $G$. 
	
	The proof in the general case is done by induction. 
	If $m\ge 2$ and one has a subsequence $\{K_{j_i}\}_{i\ge 1}$ and 
	$x^{(1)}_i,\dots,x^{(m-1)}_i\in K_{j_i}$ for every $i\ge 1$ with $\lim\limits_{i\to\infty}x^{(s)}_i=x^{(s)}$ in $G$ 
	for $s=1,\dots,m-1$, 
	then one defines $x:=x^{(m)}$ and one can select by the above method a suitable subsequence 
	$\{K_{j_{i_r}}\}_{r\ge 1}$ of 
	$\{K_{j_i}\}_{i\ge 1}$ for which there exists $x^{(m)}_r\in K_{j_{i_r}}$ for all $r\ge 1$ with 
	$\lim\limits_{r\to\infty}x^{(m)}_r=x^{(m)}$ in $G$. 
	This completes the induction step and the proof. 
\end{proof}

\begin{remark}
	\normalfont
	The proof of Lemma~\ref{points} involves neither the group structures of $G$ and $K_j$, 
	nor the fact that $G$ is locally compact. 
	Therefore its reasoning leads to a more general result: 
	Let $X$ be a topological space with its coresponding space 
	$\Xc(X)$ of all closed subsets of $X$ endowed with its Fell topology. 
	If $X$ is first countable and $\lim\limits_{j\to\infty}K_j=K$ in $\Xc(X)$, 
	then for every finite family $x^{(1)},\dots,x^{(m)}\in K$ 
	there exist 
	a subsequence $\{K_{j_i}\}_{i\ge 1}$ and 
	$x^{(1)}_i,\dots,x^{(m)}_i\in K_{j_i}$ for every $i\ge 1$ with $\lim\limits_{i\to\infty}x^{(s)}_i=x^{(s)}$ in $X$ 
	for $s=1,\dots,m$. 
\end{remark}

The following lemma uses the weak containment of subgroup representations in the sense of \cite{Fe64}.

\begin{lemma}\label{conv1}
	Let $G$ be any locally compact group which is first countable, 
	and assume that $\lim\limits_{j\to\infty}K_j=K$ in $\Kc(G)$. 
	Also let $\varphi\in\Cc(G)$ whose restriction $\varphi\vert_{K_j}$ is a character of $K_j$ for every $j\ge 1$. 
	Then the restriction $\varphi\vert_K$ is a character of $K$ 
	and the subgroup representation $\langle K,\varphi\vert_K\rangle$ is weakly contained in the family of subgroup representations 
	$\{\langle K_j,\varphi\vert_{K_j}\rangle\}_{j\ge 1}$. 
\end{lemma}

\begin{proof}
	To check that $\varphi\vert_K$ is a character of $K$ we must prove that 
	$\varphi(K)\subseteq\TT$ and $\varphi\vert_K\colon K\to\TT$ 
	is a group morphism, since $\varphi\vert_K$ is continuous by the hypothesis $\varphi\in\Cc(G)$. 
	To this end let $x,y\in K$ arbitrary. 
	It follows by Lemma~\ref{points} that, after replacing $\{K_j\}_{j\ge 1}$ by a suitable subsequence, 
	we may assume that there exist $x_j,y_j\in K_j$ for every $j\ge 1$ 
	with $\lim\limits_{j\to\infty}x_j=x$ and $\lim\limits_{j\to\infty}y_j=y$. 
	Since $\varphi\in\Cc(G)$, it then follows that 
	$\varphi(x)=\lim\limits_{j\to\infty}\varphi(x_j)$ and $\varphi(y)=\lim\limits_{j\to\infty}\varphi(y_j)$. 
	Now, using the hypothesis that $\varphi\vert_{K_j}$ is a character of $K_j$ for every $j\ge 1$, 
	it is straightforward to check that $\varphi\vert_K$ is a character of~$K$. 
	
	The weak containment assertion follows by \cite[Lemma 3.2 and Thm. 3.1']{Fe64}.
\end{proof}

\section{Topological characterization of the set of characters}
\label{Sect3}

Let  $G$ be a nilpotent Lie group. 
Then the set of characters $[\gg,\gg]^\perp$, regarded as a subset of  $\gg^*/G$,
is closed,  thus  $[\gg,\gg]^\perp\in Q(\widehat G)$. 

	For every even integer $d\ge 0$ we denote
$$(\gg^*/G)_d:=\{\Oc\in\gg^*/G\mid \dim\Oc= d\}.$$
In particular $(\gg^*/G)_0=[\gg,\gg]^\perp$. 	

\begin{lemma}\label{quot}
	Let $G$ be a connected simply connected nilpotent Lie group with its corresponding Lie algebra $\gg$. 
	If $q\colon\gg^*\to\gg^*/G$, $\xi\mapsto\Oc_{\xi}$, 
	is the quotient map onto the set of coadjoint orbits, 
	then the following assertions hold: 
	\begin{enumerate}[{\rm(i)}]
		\item 
		The complement $\widehat{G }\setminus [\gg,\gg]^\perp$ of the set of characters is either empty (if $G$ is commutative) 
		or dense in $\widehat{G}$ (if $G$ is noncommutative). 
		\item For any dense open set $D\subseteq\gg^*/G$, the set $q^{-1}(D)\subseteq\gg^*$ is also open and dense.  
	\end{enumerate}
\end{lemma}

\begin{proof}
	The stated properties follow in a 
straightforward manner by the fact that $q$ is an open continuous map. See \cite[Lemma~4.5]{BB17} for more detail. 
\end{proof}

The following result is a version of \cite[Lemma 6.8(2)]{BBL17} for coadjoint orbits that are not necessarily flat.

\begin{lemma}\label{30december2017}
	Let $G$ be a 
	nilpotent Lie group.  
For any $S \in Q(\widehat G)$, the set $S\cap [\gg, \gg]^\perp$ is both closed and open in $S$.
	\end{lemma}

\begin{proof} 
	It suffices to prove 
	the set $S\cap [\gg, \gg]^\perp $ is simultaneously 
	closed and open with respect to the relative topology of~$S$. 
	The set $[\gg, \gg]^\perp$ is a closed subset of $\gg^*/G$ therefore $S\cap [\gg, \gg]^\perp$ is a relatively closed subset of~$S$. 
	Thus it remains to prove that $S\setminus [\gg, \gg]^\perp$ is a relatively closed subset of~$S$. 
	
	Assume that $S\setminus [\gg, \gg]^\perp$ is not a relatively closed subset of~$S$, hence 
	 there exist
	a sequence $\{\xi_j\}_{j\in\NN}$ in $\gg^*$ and $\eta\in\gg^*$ with $\lim\limits_{j\in\NN}\xi_j=\eta$ and $\Oc_{\xi_j}\in S\setminus [\gg, \gg]^\perp$ for every $j\in\NN$ while 
	$\Oc_\eta =\{\eta\} \in S\cap [\gg, \gg]^\perp$. 
	
	Selecting a suitable subsequence, we may assume that
	there exists an integer  $d>0$ with $\dim\Oc_{\xi_j}=d$ for every 
	$j\in \NN$. 
	Selecting again a suitable subsequence, using that the Grassmann manifold $\Gr(\gg)$ is compact, we may assume that for every $j\in \NN$ there exists a polarization $\pg_j\subseteq\gg$ at $\xi_j\in\gg^*$ such that the limit $\mg:=\lim\limits \pg_j$ exists in~$\Gr(\gg)$.
	The subspace $\mg$ is a subalgebra of $\gg$ subordinated to $\eta$.
	If we denote by $P_j$ and $M$ the corresponding closed subgroups of $G$ with Lie algebras $\pg_j$ and $\mg$ respectively, 
	we have that $\lim P_j =M$ in $\Kc(G)$. (See \cite{BB18}.) 
	It follows by Lemma~\ref{conv1} and \cite[Thm. 4.3]{Fe64}
	that the set $L$ of limit points of the sequence $\{\Oc_{\xi_j}\}_{j\in \NN}$ contains the coadjoint orbits of all irreducible representations of $G$ that occur in the disintegration of the representation 
	$\mathop{\rm Ind}_M^G(\ee^{\ie\eta}\vert_M)$. 
	By \cite[Thm.~2 and p.~556]{CGG87} we thus get that
	$L\supseteq L_0:=q(\eta+\mg^\perp) $  and $\Oc_\eta \in L_0$.
	 Since $\dim\Oc_\eta=0$, the affine subspace $\eta+\mg^\perp$  is not contained in the coadjoint orbit of $\eta$,
	 hence $L_0$ is a connected set that contains more than one point. 
	 Thus $\eta$ is not an isolated point in $L$. 
	 On the other hand, since $S$ is a closed subset of $\widehat{G}$, it follows that $L\subseteq S$.
	We thus obtained a contradiction with the assumption that the relative topology of $S$ is Hausdorff, and this completes the proof.  
\end{proof}

	\begin{proposition}\label{maximal}
		For every  
		nilpotent Lie group~$G$,   
		the set of characters $[\gg,\gg]^\perp$ is a maximal element of $Q(\widehat G)$. 
	\end{proposition}
	
	\begin{proof}
		If $S\in Q(\widehat G)$ and $[\gg,\gg]^\perp\subseteq S$, 
		then $[\gg,\gg]^\perp$ is a closed-open subset of $S$ by Lemma~\ref{30december2017}, since $[\gg,\gg]^\perp=(\gg^*/G)_d$ for $d=0$. 
		Using the fact that $S$ is connected and $[\gg,\gg]^\perp\ne\emptyset$ we then obtain $[\gg,\gg]^\perp=S$. 
		Hence $[\gg,\gg]^\perp$ is a maximal element of $Q(\widehat G)$. 
	\end{proof}
	
	\begin{remark}\label{order}
		\normalfont
		The set $Q(\widehat G)$ is not inductively ordered, as could be seen for instance  in the case when  $G$ is a Heisenberg group. 
		Thus not every element of $Q(\widehat G)$ is necessarily included in a maximal element.
		\end{remark}

	For the nilpotent Lie groups for which all coadjoint orbits are flat,  
	the result of Proposition~\ref{maximal} can be improved 
	by showing that $[\gg,\gg]^\perp$ is the unique maximal element of $Q(\widehat G)$; see Proposition~\ref{P3}. 
	That uniqueness property of $[\gg,\gg]^\perp$ turns out to be shared by some other groups (see Lemma~\ref{L5}). 
	It is then convenient to introduce the following terminology   
	that is motivated by Lemma~\ref{L2} below and by the formula for the real rank of group $C^*$-algebras 
	\begin{equation}\label{RR_eq}
	\RRa (C^*(G))=\dim[\gg,\gg]^\perp
	\end{equation}
	which holds for all exponential solvable Lie groups $G$, as shown in  \cite[Thm.~3.5]{BB16b}. 
(See also \cite{BB17}.)

\begin{definition}\label{D1}
	\normalfont
	Let $G$ be a nilpotent Lie group. 
	\begin{enumerate}[{\rm(a)}]
		\item\label{D1_item1} One says that $G$ is \emph{type~$\RRa$} if for every other nilpotent Lie group $H$ for which $C^*(G)$ and $C^*(H)$ are Morita equivalent, one has $\dim[\gg,\gg]^\perp=\dim[\hg,\hg]^\perp$. 
		\item\label{D1_item2} One says that $G$ is \emph{type~$\RRa_1$} if $[\gg,\gg]^\perp$ is the unique maximal element of $Q(\widehat G)$. 
	\end{enumerate}
\end{definition}

\begin{lemma}\label{L2}
	Let $G_1$ be nilpotent Lie group and assume that 
	$G_1$ is type~$\RRa_1$. 
	If~$G_2$ is another nilpotent Lie group and one has a homeomorphism
	$\psi\colon\widehat{G}_1\to\widehat{G}_2$,  
	then $\psi([\gg_1,\gg_1]^\perp)=[\gg_2,\gg_2]^\perp$ 
	and $\dim[\gg_1,\gg_1]^\perp=\dim[\gg_2,\gg_2]^\perp$.
	In particular,  if a nilpotent Lie group $G$ is type~$\RRa_1$, then it is type~$\RRa$.
\end{lemma}

\begin{proof}
	The mapping 
	$Q(G_1)\to Q(G_2)$, $S\mapsto \psi(S)$, 
	is an isomorphism of partially ordered sets, 
	since $\psi\colon\widehat{G}_1\to\widehat{G}_2$ is a homeomorphism.
	Hence $S\in Q(G_1)$ is a maximal element if and only if 
	$\psi(S)$ is a maximal element of $Q(G_2)$. 
	
	Since $G_1$ is type~$\RRa_1$, 
	it follows
	that the set $Q(G_1)$ has 
	exactly one maximal element, namely $[\gg_1,\gg_1]^\perp$. 
	Therefore the set $Q(G_2)$ in turn has exactly one maximal element, namely $\psi([\gg_1,\gg_1]^\perp)$. 
	On the other hand, we know from Proposition~\ref{maximal} that $[\gg_2,\gg_2]^\perp$ is a maximal element of $Q(G_2)$, 
	hence $\psi([\gg_1,\gg_1]^\perp)=[\gg_2,\gg_2]^\perp$. 
	Along with the fact that $\psi$ is a homeomorphism, 
	this further implies that $\psi\vert_{[\gg_1,\gg_1]^\perp}\colon [\gg_1,\gg_1]^\perp\to[\gg_2,\gg_2]^\perp$ is a homeomorphism. 
	Moreover,  $[\gg_1,\gg_1]^\perp$ and $[\gg_2,\gg_2]^\perp$ are vector spaces, hence it  follows by Brouwer's theorem on invariance of domain that $\dim[\gg_1,\gg_1]^\perp=\dim[\gg_2,\gg_2]^\perp$.

	For the second assertion, if $H$ is a nilpotent Lie group for which $C^*(G)$ is Morita equivalent to $C^*(H)$, then $\widehat{G}$ and $\widehat{H}$ are homeomorphic by \cite[Cor. 3.33]{RaWi98}, 
	hence $\dim[\gg,\gg]^\perp=\dim[\hg,\hg]^\perp$ by Lemma~\ref{L2}. 
	This concludes the proof.
\end{proof}

\begin{proposition}\label{P3}
	If $G$ is a nilpotent Lie group for which all its coadjoint orbits are flat, then $G$ is type~$\RRa_1$. 
\end{proposition}

\begin{proof} 
		It follows by Proposition~\ref{maximal} that $[\gg,\gg]^\perp$ is a maximal element of $Q(\widehat G)$. 
		Let $S$ be an arbitrary maximal element of $Q(\widehat G)$.  
		We prove that $S\subseteq[\gg,\gg]^\perp$ and then,  the maximality of $S$ implies that  $S=[\gg,\gg]^\perp$. 
		
		Since $S$ is connected, it follows by 
		\cite[Lemma 6.8(2)]{BBL17} 
		that there exists an even integer $d\ge 0$ such that $S$ is contained in the set $(\gg^*/G)_d$ of all $d$-dimensional coadjoint orbits. 
		We may assume $S\ne\emptyset$ and then we may select $\xi\in\gg^*$ with $\Oc_\xi\in S$. 
		Let us define the continuous path $\gamma\colon\RR\to \gg^*/G$, 
		$\gamma(t):=\Oc_{t\xi}$. 
		One has $\gamma(\RR^\times)\subseteq (\gg^*/G)_d$. 
		
		Since $S$ is closed, it follows that 
		$A:=\{t\in(0,\infty)\mid \gamma(t)\in S\}$ is a closed subset of $(0,\infty)$. 
		For $t_0:=\inf A$ one has $t_0\le 1$ since $\gamma(1)=\Oc_\xi\in S$. 
		
		If $t_0>0$ then we define $S_0:=\gamma([t_0/2,1])$, 
		which is a compact subset of  $(\gg^*/G)_d$,  
		since the relative topology of $(\gg^*/G)_d$ is Hausdorff by \cite[Lemma 6.8(1)]{BBL17}. 
		In particular $S_0$ is a closed subset of $\gg^*/G$, 
		and it is clear that $S_0$ is connected. 
		Moreover, $\gamma(1)\in S_0\cap S$, hence  
		$S_0\cup S$ is a closed connected subset of $\gg^*/G$.  
		Since $S_0\cup S\subseteq (\gg^*/G)_d$, 
		the relative topology of $S_0\cup S$ is Hausdorff,  
		and thus $S_0\cup S\in Q(S)$. 
		On the other hand, by the definition of $t_0$ one has $S_0\not\subseteq S$, 
		hence $S\not\subseteq S_0\cup S$, 
		which is a contradiction with the hypothesis that $S$ is a maximal element of $Q(\widehat G)$. 
		
		Consequently $t_0=0$, and it follows that  there exists a sequence $\{t_n\}_{n\in\NN}$ in $(0,\infty)$ with $\lim\limits_{n\in\NN}t_n=0$ and $\Oc_{t_n\xi}=\gamma(t_n)\in S$ for every $n\in\NN$. 
		Since $\lim\limits_{n\in\NN}t_n\xi=0$, 
		the coadjoint orbit $\Oc_0=\{0\}$ is a limit point of
		the sequence $\{\Oc_{t_n\xi}\}_{n\in\NN}$.
		The set $S$ is a closed subset of 
		$\gg^*/G$,
		and thus $\Oc_0\in S$. 
		We noted above that $S\subseteq(\gg^*/G)_d$, 
		hence $d=0$, that is, $S\subseteq[\gg,\gg]^\perp$.
		This concludes the proof.
\end{proof}

	It follows by \eqref{RR_eq} and Definition~\ref{D1}\eqref{D1_item1} that real rank of $C^*$-algebras of nilpotent Lie groups of type~$\RRa$ is preserved by Morita equivalence. 
	As it is well known, this not the case for arbitrary $C^*$-algebras (see for instance \cite[Lemma~3.1 and Rem.~2.2]{BB16b}).

	\section{Heisenberg groups are uniquely determined by their spectrum}\label{proof-th}

	The present section is devoted to the
proof of Theorem~\ref{2H}. 
	\begin{lemma}\label{GCR}
		Let $G$ be an exponential solvable Lie group.
		Then the following assertions are equivalent:
		\begin{enumerate}[{\rm (i)}]
			\item $G$ is a liminary group.
			\item Every coadjoint orbit of $G$ is a closed subset of $\gg^*$.
			\item $G$ is a nilpotent Lie group.
			\end{enumerate}
		\end{lemma}
	
	\begin{proof}
		The group $G$ is  exponential, hence 
		it is a connected, simply connected, solvable Lie group of type~I 
		(see \cite[Sect. 0, Rem. 3]{AuKo71}).
		Then  Glimm's characterization of separable $C^*$-algebras of type~I (see \cite[\S 9.1]{Dix77})
		implies that $G$  is postliminary.
		By \cite[Ch. V, Thm. 1--2]{AuMo66} we get that the first assertion is equivalent with the fact that $G$ is of  type R, that is, for every $x\in\gg$, all the eigenvalues of the
		operator $\ad_{\gg} x\colon\gg_{\CC}\to\gg_{\CC}$
		are purely imaginary.
		
		On the other hand since $G$ is an
		exponential Lie group,
		it follows that for every $x\in\gg$, the operator 
		$\ad_G x\colon\gg_{\CC}\to\gg_{\CC}$
		has no nonzero purely imaginary eigenvalues (see \cite[Prop. 5.2.13, Thm. 5.2.16]{FuLu15}),
		hence the first and third assertion are equivalent.
		Finally, the equivalence between the first and second assertion follows 
		from \cite[Thm. 5.3.31]{FuLu15} and \cite[Ex.~9.5.3]{Dix77}.
		\end{proof}
	
	The next lemma treats the case of abelian Lie groups. 
	
	\begin{lemma}\label{abelian}
		Let $G_1$ be an abelian Lie group. 
		If $G_2$ is an exponential Lie group with $\widehat{G_1}$ homeomorphic to  $\widehat{G_2}$, then the Lie groups $G_1$ and $G_2$ are isomorphic.
		\end{lemma}
	
	\begin{proof}
	Since 	 $\widehat{G_1}$ homeomorphic to  $\widehat{G_2}$, it follows that 
	$\widehat{G_2}$ must be Hausdorff.
	Therefore, by \cite[Thm.~1]{BS81}, the group $G_2$ must be a compact extension of vector group. Since $G_2$ is assumed to be exponential, it cannot have non-trivial compact subgroups, hence it is a vector group. 
	
	Thus, the duals of $G_1$ and $G_2$ are vector spaces of dimensions $\dim G_1$ and $\dim G_2$, respectively. 
	By Brouwer's theorem on the invariance of domain we then obtain that
	$\dim G_1=\dim G_2$, hence $G_1$ and $G_2$ are isomorphic. 
	\end{proof}

	\begin{proof}[Proof of Theorem~\ref{2H}]
		We may assume by Lemma~\ref{abelian} that $G_1$ is not abelian, 
		therefore $\dim[\gg_1,\gg_1]=1$. 
		
		Assume that $\psi\colon\widehat{G}_1\to\widehat{G}_2$ is a homeomorphism.  
		Since $G_1$ is a nilpotent Lie group, 
		the singleton subsets of $\widehat{G}_1$ are closed, 
		hence the singleton subsets of $\widehat{G}_2$ are also closed. 
		The group $G_2$ is an exponential Lie group, thus  
		Lemma~\ref{GCR} implies now that $G_2$ is a nilpotent Lie group. 
		
		Since $\dim[\gg_1, \gg_1]=1$ it follows that there exist some integers 
		$d_1\ge 1$ and $k_1\ge 0$ with $G_1=A_{k_1}\times H_{2d_1+1}$.
		Then all the coadjoint orbits of $G_1$ are flat and we then obtain by Proposition~\ref{P3} and Lemma~\ref{L2} 
		that 
		$\psi([\gg_1,\gg_1]^\perp)=[\gg_2,\gg_2]^\perp$.  
		In particular, for 
		$\Gamma^{(j)}_1:=\widehat{G}_j\setminus[\gg_j,\gg_j]^\perp$ 
		we also obtain $\psi(\Gamma^{(1)}_1)=\Gamma^{(2)}_1$. 
		Moreover, 
		$\widehat{G_1}=\widehat{A_{k_1}}\times\widehat{H_{2d_1+1}}$, 
		and $\Gamma^{(1)}_1=\widehat{A_{k_1}}\times\RR^\times$ (up to a homeomorphism), 
		which is a disconnected topological space. 
		Therefore the set $\Gamma^{(2)}_1$ must be disconnected as well. 
		
		Assume now that  $\dim[\gg_2,\gg_2]\ge 2$. 
		Then $\gg_2^*\setminus[\gg_2,\gg_2]^\perp$ is connected, 
		hence its image through the quotient map $q\colon\gg_2^*\to\gg_2^*/G_2$ is connected. 
		That is, the set $\Gamma^{(2)}_1$ is connected, 
		which is a contradiction. 
		Consequently $\dim[\gg_2,\gg_2]\le 1$, and then there exist some integers 
		$k_2,d_2\ge 0$ with $G_2=A_{k_2}\times H_{2d_2+1}$ (up to a group isomorphism). 
		One actually has $d_2\ge 1$ since $\Gamma^{(2)}_1=\psi(\Gamma^{(1)}_1)\ne\emptyset$. 
		One then has $\Gamma^{(2)}_1=\widehat{A_{k_2}}\times\RR^\times$ and  this set is homeomorphic to $\widehat{A_{k_1}}\times\RR^\times$. 
		Thus we obtain $k_1=k_2$ by the theorem on invariance of domain. 
		Since $\dim[\gg_j,\gg_j]^\perp=2d_j+k_k$, we then also obtain $d_1=d_2$, and this completes the proof. 
	\end{proof}

	\section{Application to nilpotent Lie groups of dimension $\le5$}\label{proof-cor}
	
This section is devoted to the proof of Theorem~\ref{5D}.

	\subsection{Nilpotent Lie groups of whose all nontrivial coadjoint orbits have the same dimension}

	In order to prepare for the proof of Theorem~\ref{5D}, we study here a class of more tractable nilpotent Lie groups, namely those 
	whose all nontrivial coadjoint orbits have the same, maximal, dimension. 
	We make first a definition.
	
	\begin{definition}\label{def_Tc}
		\normalfont
		A  nilpotent Lie algebra $\gg$ is called \emph{of class~$\Tc$} if 
		for all $\xi\in\gg^*\setminus[\gg,\gg]^\perp$ one has $\dim\gg(\xi)=\ind\gg=\min\{\dim \gg(\eta) \mid \eta \in \gg^*\}$. 
		(See also Lemma~\ref{27december2017-1555}.)
		If this is the case, then the nilpotent Lie group~$G$ is called \emph{of class  $\Tc$}. 
	\end{definition}

	\begin{lemma}\label{8feb2018}
		Let $\gg$ be a finite-dimensional real Lie algebra and let $\Jc(\gg)$ be the set of its ideals, that is, 
		$$\Jc(\gg):=\{\hg\in\Gr(\gg)\mid [\gg,\hg]\subseteq\hg\}.$$
		Then $\Jc(\gg)$ is a closed subset of $\Gr(\gg)$. 
	\end{lemma}
	
	\begin{proof}
		We prove that for every sequence $\{\hg_n\}_{n\in\NN}$  in $\Jc(\gg)$ for which there exists $\hg=\lim\limits_{n\in\NN}\hg_n$ in $\Gr(\gg)$, 
		one has $\hg\in\Jc(\gg)$. 
		To this end let $x\in\gg$ and $y\in\hg$ arbitrary. 
		Then for every $n\in\NN$ there exists $y_n\in\hg_n$ with 
		$y=\lim\limits_{n\in\NN}y_n$, hence 
		$[x,y]=\lim\limits_{n\in\NN}[x,y_n]$. 
		For every $n\in\NN$ one has $\hg_n\in\Jc(\gg)$, hence $[x,y_n]\in\hg_n$, and then $[x,y]\in\lim\limits_{n\in\NN}\hg_n=\hg$. 
		Thus $[\gg,\hg]\subseteq\hg$, and we are done. 
	\end{proof}

	\begin{lemma}\label{9feb2018}
		Let $\gg$ be a nilpotent Lie algebra. 
		Assume that there exist two subsets $A_0,A\subseteq\gg^*$ and an even natural number $d\in\NN$ satisfying the following conditions: 
		\begin{enumerate}[{\rm(a)}]
			\item\label{9feb2018_itema} For every $\xi\in A$ one has $\dim\Oc_\xi=d$. 
			\item\label{9feb2018_itemb} One has $A\subseteq\overline{\bigcup\limits_{\xi\in A_0}\Oc_\xi}$. 
			\item\label{9feb2018_itemc} For every $\xi\in A_0$ one has $\Oc_\xi=\xi+\gg(\xi)^\perp$. 
		\end{enumerate}
		Then the following assertions hold: 
		\begin{enumerate}[{\rm(i)}]
			\item\label{9feb2018_item1} For every $\xi\in A$ one has $\Oc_\xi=\xi+\gg(\xi)^\perp$. 
			\item\label{9feb2018_item2} If $A=\gg^*\setminus[\gg,\gg]^\perp$, then the Lie algebra $\gg$ is 2-step nilpotent. 
		\end{enumerate}
	\end{lemma}
	
	\begin{proof}
		\eqref{9feb2018_item1} 
		It follows by \eqref{9feb2018_itema} along with \cite[Lemma 5.3]{BB18} that the mapping 
		$$\psi\colon A\to\Gr(\gg), \quad \xi\mapsto\gg(\xi),$$ 
		is continuous. 
		Now let us denote $B_0:=\bigcup\limits_{\xi\in A_0}\Oc_\xi\subseteq\gg^*$. 
		By \cite[Thm. 3.2.3]{CG90}, the hypothesis \eqref{9feb2018_itemc} is equivalent to $\psi(A_0)\subseteq\Jc(\gg)$. 
		This is easily seen to be further equivalent to $\psi(B_0)\subseteq\Jc(\gg)$, which implies, by Lemma~\ref{8feb2018}, 
		$\overline{\psi(B_0)}\subseteq\Jc(\gg)$. 
		Therefore, using hypothesis~\eqref{9feb2018_itemb} and the continuity of $\psi$, one obtains 
		$$\psi(A)\subseteq\psi(\overline{B_0})\subseteq \overline{\psi(B_0)}\subseteq\Jc(\gg).$$ 
		The assertion then follows by \cite[Thm. 3.2.3]{CG90} again.

		\eqref{9feb2018_item2}
		Using the notion of cortex $\cortex(\gg^*)$ with respect to the coadjoint action of $G$ defined as in \cite[Ch. III]{Bk95} and \cite{Bk97}, we first prove that 
		\begin{equation}\label{9feb2018_proof_eq1}
		\cortex(\gg^*)\subseteq[\gg,\gg]^\perp. 
		\end{equation}
		Indeed, the set $\gg^*\setminus [\gg, \gg]^\perp$ consists of flat coadjoint orbits of maximal dimension of $G$, hence  by \cite[Thm.~2.2]{ArSoKaSc99}, 
		they are separated points in $\gg^*/G$. Thus \eqref{9feb2018_proof_eq1}
		follows.

		It follows by \cite[Thm. III-3.3]{Bk95} along with \eqref{9feb2018_proof_eq1} that 
		for all $\xi\in\gg^*$ and $x\in\gg$ one has $(\ad^*_{\gg}x)\xi\in\cortex(\gg^*)\subseteq[\gg,\gg]^\perp$. 
		Therefore for all $\xi\in\gg^*$ and $x,y,z\in\gg$ 
		one has $\langle\xi,[x,[y,z]]\rangle=0$, 
		which directly implies that the Lie algebra $\gg$ is 2-step nilpotent. 
	\end{proof}
	
	
	\begin{proposition}\label{MD3}
		If $G$ is a nilpotent Lie group of class~$\Tc$, and $Z$ is the center
		of $G$, then the following assertions hold.
		\begin{enumerate}[{\rm(i)}]
			\item\label{MD3_item1}
			One has $\ind G=1$ if and only if $G$ is a Heisenberg group.
			\item\label{MD3_item2}
			If $\ind G=2$, then either $G$ is 2-step nilpotent, or one has
			$\dim\zg=1$, $\dim [\gg, \gg]=2$, and there exists $\xi\in\gg^*$ with
			$\gg(\xi)=[\gg, \gg]$.
		\end{enumerate}
	\end{proposition}
	
	\begin{proof}
		\eqref{MD3_item1}
		By Lemma~\ref{27december2017-1555} one obtains $\dim\zg=1$ and there
		exists $\xi_0\in\gg^*$ with $\gg_2(\xi_0)=\zg$ and for which
		$\Gamma:=\{\Oc_{t\xi_0}\mid t\in\RR^\times\}$ is an open dense subset
		of $\gg^*/G$.
		Then, defining $q\colon\gg^*\to \gg^*/G$, $\xi\mapsto\Oc_\xi$,
		it follows by Lemma~\ref{quot} that $q^{-1}(\Gamma)$ is dense in
		$\gg^*$, that is, $\bigcup\limits_{t\in\RR^\times}\Oc_{t\xi_0}$ is
		dense in~$\gg^*$.
		Moreover, since $\gg(t\xi_0)=\zg$ for all $t\in\RR^\times$,
		it follows by \cite[Thm. 3.2.3]{CG90} that $\Oc_{t\xi_0}=t\xi_0+\zg^\perp$.
		Hence we may use Lemma~\ref{9feb2018}\eqref{9feb2018_item2} with
		$A_0=\{t\xi_0\mid t\in\RR^\times\}$ to obtain that the Lie algebra
		$\gg$ is 2-step nilpotent.
		Since we have noted above that the center of $\gg$ is 1-dimensional,
		it then follows that $\gg_2$ is a Heisenberg algebra.
		\\
		\eqref{MD3_item2}
		Let us assume that $\ind\gg=2$ and $\gg$ is not 2-step nilpotent, that
		is, $[\gg, \gg]\not\subseteq\zg$.
		In particular $[\gg, \gg]\ne\{0\}$, and then there exists
		$\xi\in\gg^*\setminus([\gg, \gg]^\perp)$ with $[\gg, \gg]\subseteq\gg(\xi)$, by
		Lemma~\ref{MD1}.
		Since $\xi\not\in[\gg, \gg]^\perp$ and $\gg$ is of class~$\Tc$,
		one has $\dim\gg(\xi)=\ind\gg=2$.
		One also has $\zg\subseteq\gg(\xi)$ and $[\gg, \gg]\not\subseteq\zg$,
		hence necessarily $\dim\zg=1$.
		Then $\zg\subseteq[\gg, \gg]$.
		If $\zg=[\gg, \gg]$, then $\gg$ is a Heisenberg algebra, which is a
		contradiction with the assumption that $\gg$ is not 2-step nilpotent.
		Consequently $\zg\subsetneqq[\gg, \gg] \subseteq\gg(\xi)$.
		Since $\dim\gg(\xi)=2$, the assertion follows.
	\end{proof}

	\begin{example}\label{Ex5.5}
		\normalfont
		There are many examples of 2-step nilpotent Lie algebras satisfying the hypotheses of Lemma~\ref{9feb2018}\eqref{9feb2018_item2}, 
		completely different from the Heisenberg algebras. 
		Here is a list of examples that illustrate this assertion.
		\begin{enumerate}
			\item\label{Ex5.5_item1} For any finite-dimensional real vector space $\Vc$ and any  $D\in\End(\Vc)$ with $D^2=0$, $D\ne 0$,  define the Lie algebra  $\gg=\gg_D:=\Vc\rtimes_{\alpha_D}\RR$, using notation from \cite[Sect. 2]{BB18}. 
			That is, $\gg_D=\Vc\dotplus\RR$ as a vector space, and the Lie bracket of~$\gg_D$ is given by $[(v_1,t_1),(v_2,t_2)]=(t_1Dv_2-t_2Dv_1,0)$ for all $v_1,v_2\in\Vc$ and $t_1,t_2\in\RR$. 
			
			Then the Lie algebra $\gg$ is 2-step nilpotent since $D^2=0$, and $\dim\Oc_\xi=2$ for all $\xi\in\gg^*\setminus[\gg,\gg]^\perp$. 
			(See for instance \cite[Thm. (ii)]{ACL95}.) 
			This includes for instance 
			the Lie algebra defined by a basis $X_1,X_2,X_3,X_4,X_5$ satisfying  
			$[X_5,X_4]=X_2$, $[X_5,X_3]=X_1$, 
			studied in \cite[N5N2]{Pe88}. 
			
			\item\label{Ex5.5_item2} Let $\gg$ be the so-called free 2-step nilpotent Lie algebra with 3 generators, defined by a basis $X_1,X_2,X_3,X_4,X_5,X_6$ satisfying  
			$[X_6,X_5]=X_3$, $[X_6,X_4]=X_1$, $[X_5,X_4]=X_2$.   
			This satisfies $\dim\Oc_\xi=2$ for all $\xi\in\gg^*\setminus[\gg,\gg]^\perp$, and was studied in \cite[Thm. (iv)]{ACL95} and  \cite[N6N15]{Pe88}. 
			
			\item\label{Ex5.5_item3} Let $\gg$ be the 2-step nilpotent Lie algebra with a basis $X_1,X_2,X_3,X_4,X_5,X_6$ satisfying  
			$[X_6,X_5]=X_2$, $[X_6,X_3]=X_1$, $[X_5,X_4]=X_1$, $[X_4,X_3]=X_2$.   
			Then one has $\dim\Oc_\xi=4$ for all $\xi\in\gg^*\setminus[\gg,\gg]^\perp$. 
			(See \cite[N6N17]{Pe88}.)
		\end{enumerate}
		
	\end{example}
	
	\begin{example}
		\normalfont
		Here we give some counterexamples related to Lemma~\ref{9feb2018}\eqref{9feb2018_item2}. 
		\begin{enumerate} 
			\item If $\gg$ is a 2-step nilpotent Lie algebra, then 
			there may exist no number $d\in\NN$ with $\dim\Oc_\xi=d$ for all $\xi\in\gg^*\setminus[\gg,\gg]^\perp$. 
			For instance, 
			the 8-dimensional Lie algebra $\gg=\hg_3\times\hg_5$ (where $\hg_{2k+1}$ denotes the $(2k+1)$-dimensional Heisenberg algebra for any $k\in\NN$) is 2-step nilpotent and yet it has both 2-dimensional and 4-dimensional coadjoint orbits.
			See also Lemmas~\ref{MD-indecomp} and \ref{MD-red}. 
			\item If $\gg$ be a nilpotent Lie algebra for which there exists $d\in\NN$ with $\dim\Oc_\xi=d$ for all $\xi\in\gg^*\setminus[\gg,\gg]^\perp$,  
			then it does not necessarily follow that $\gg$ is 2-step nilpotent. 
			This can be proved by several examples: 
			\begin{itemize}
				\item If $\Vc$ is a finite-dimensional real vector space, 
				$D\in\End(\Vc)$ is a nilpotent operator with $D^2\ne0$, 
				and we define  
				 $\gg=\gg_D:=\Vc\rtimes_{\alpha_D}\RR$ 
				as in Example~\ref{Ex5.5}\eqref{Ex5.5_item1} above, 
				then $\gg$ is not 2-step nilpotent and yet  $\dim\Oc_\xi=2$ for all $\xi\in\gg^*\setminus[\gg,\gg]^\perp$. 
				(See for instance \cite[Thm. (ii)]{ACL95}, which includes \cite[N6N18]{Pe88}.)
				\item If $\gg$ is the Lie algebra defined by a basis $X_1,X_2,X_3,X_4,X_5$ satisfying  
				$[X_5,X_4]=X_3$, $[X_5,X_3]=X_2$, $[X_4,X_3]=X_1$, 
				then $\gg$ is 3-step nilpotent and yet  $\dim\Oc_\xi=2$ for all $\xi\in\gg^*\setminus[\gg,\gg]^\perp$.  (See \cite[Thm. (v)]{ACL95} or \cite[N5N4]{Pe88}.)
			\end{itemize} 
		\end{enumerate}
	\end{example}

	\begin{lemma}\label{MD1}
		Let $\gg$ be a Lie algebra with $[\gg, \gg]\ne\{0\}$. Then there exists $\xi\in\gg^*\setminus([\gg, \gg]^\perp)$ with $[\gg, \gg]\subseteq\gg(\xi)$, 
		and $\gg(\xi)$ is an ideal of $\gg$. 
	\end{lemma}
	
	\begin{proof}
		Denote $\gg^1=[\gg, \gg]$.
		Since $\gg$ is nilpotent, one has $[\gg,\gg^1]\subsetneqq\gg^1$. 
		Then there exists $\xi\in\gg^*$ with $[\gg,\gg^1]\subseteq\Ker\xi$ and $\gg^1\not\subseteq\Ker\xi$. 
		That is, $\gg^1\subseteq\gg(\xi)$ and $\xi\not\in(\gg^1)^\perp$. 
		
		Since $\gg^1\subseteq\gg(\xi)$, it follows that $\gg(\xi)$ is an ideal of  $\gg$, 
		and this completes the proof. 
	\end{proof}

	\begin{lemma}\label{MD2}
		If the nilpotent Lie algebra $\gg$ is of class $\Tc$ then 
		$$\dim\gg\le\ind\gg+\dim [\gg, \gg]^\perp.$$
	\end{lemma}
	
	\begin{proof}
		One has $\dim[\gg, \gg]^\perp=\dim\gg-\dim[\gg, \gg]$.
		Hence the assertion is equivalent to 
		\begin{equation}\label{MD2_proof_eq1}
		\dim[\gg, \gg]\le\ind\gg.
		\end{equation}
		To prove this inequality, we may assume $[\gg, \gg]\ne\{0\}$. 
		Then, by Lemma~\ref{MD1}, there exists 
		$\xi\in\gg^*\setminus([\gg, \gg]^\perp)$ with $[\gg, \gg]\subseteq\gg(\xi)$, 
		hence $\dim[\gg, \gg]\le\dim\gg(\xi)$. 
		Since $\xi\not\in[\gg, \gg]^\perp$ and $\gg$ is of class $\Tc$, 
		one has $\dim\gg(\xi)=\ind\gg$, 
		and thus \eqref{MD2_proof_eq1} follows.   
	\end{proof}
	
	\begin{proposition}\label{P4}
		Let $G_1$ and $G_2$ be nilpotent Lie groups such that  $C^*(G_1)$ and $C^*(G_2)$ are Morita equivalent. 
		Then $G_1$ is of class $\Tc$ if and only if $G_2$ is of class $\Tc$, and if this is the case then 
		for every Rieffel homeomorphism 
		$\psi\colon\widehat{G}_1\to\widehat{G}_2$  
		one has $\psi([\gg_1,\gg_1]^\perp)=[\gg_2,\gg_2]^\perp$ 
		and $\dim[\gg_1,\gg_1]^\perp=\dim[\gg_2,\gg_2]^\perp$. 
		
		In particular, if $G$ is a nilpotent Lie group of class~$\Tc$, then $G$ is type~$\RRa$. 
		\end{proposition}
	
	\begin{proof}
		Assume that $G_1$ is of class $\Tc$, and let $\psi\colon\widehat{G}_1\to\widehat{G}_2$ be a Rieffel homeomorphism, which exists by \cite[Cor. 3.33]{RaWi98}. 
		
		Since $G_1$ is of class $\Tc$, the open subset
		$$D_1:=\widehat{G}_1\setminus[\gg_1,\gg_1]^\perp$$ 
		of its unitary dual corresponds (via Kirillov's correspondence) to the coadjoint orbits of $G_1$ having maximal dimension. 
		It then follows by \cite[Cor. 2.9]{ArSoSp97} that $D_1$ is exactly the set of all points in $\widehat{G}_1=\widehat{C^*(G_1)}$ that have finite upper multiplicities. 
		By \cite[Cor. 5.4]{ArSoSp97} or \cite[Cor. 13(2)]{aHRaWi07}, the set $\psi(D_1)$ 
		consists of the points of $\widehat{C^*(G_2)}$ that have finite upper multiplicities. 
		Then, by \cite[Cor. 2.9]{ArSoSp97} again, $\psi(D_1)$ corresponds to the set of coadjoint orbits of $G_2$ having maximal dimension.
		
		In particular, one has 
		\begin{equation}\label{MD4bis_proof_eq1}
		[\gg_2,\gg_2]^\perp\subseteq\widehat{G_2}\setminus\psi(D_1).
		\end{equation}
		Since $\psi$ is a bijection, one has $\widehat{G_2}\setminus\psi(D_1)	
		=\psi(\widehat{G_1}\setminus D_1)
		=\psi([\gg_1,\gg_1]^\perp)$. 
		Then, using $[\gg_1,\gg_1]^\perp\in Q(G_1)$, one obtains 
		$\widehat{G_2}\setminus\psi(D_1)\in Q(G_2)$. 
		On the other hand, $[\gg_2,\gg_2]^\perp$ is a maximal element of $Q(G_2)$ by Proposition~\ref{maximal} hence, by~\eqref{MD4bis_proof_eq1}, one obtains 
		$$[\gg_2,\gg_2]^\perp=\widehat{G_2}\setminus\psi(D_1).$$ 
		We have proved above that  $\psi(D_1)$ corresponds to the set of coadjoint orbits of $G_2$ having maximal dimension, 
		hence it follows that $G_2$ is class~$\Tc$ and 
		$\psi([\gg_1,\gg_1]^\perp)=[\gg_2,\gg_2]^\perp$. 
		This equality also implies $\dim[\gg_1,\gg_1]^\perp=\dim[\gg_2,\gg_2]^\perp$ 
		by Brouwer's theorem on invariance of domain, and we are done. 
	\end{proof}
	
	

	\subsection{A weaker version of Theorem~\ref{5D}}
	
	In this subsection we prove, among other things, that two nilpotent Lie groups of dimension $\le 5$ are isomorphic if and only if their 
	unitary dual spaces are homeomorphic
	(Proposition~\ref{P7}).  
	This will help us to establish in Proposition~\ref{6D-MD} 
	a partial version of Theorem~\ref{5D}.
	
	We first recall the classification of nilpotent Lie algebras of dimension $\le 5$ over~$\RR$,  
	and to this end we need to introduce some notation. 
	Unless otherwise mentioned, $X_1,\dots,X_n$ is a basis of a nilpotent Lie algebra of dimension $n\le 5$, and we give only the brackets $[X_j,X_k]$ that are different from zero. 
	
	We consider the following nilpotent Lie algebras:
	\begin{enumerate}[(a)]
		\item {\it Case $n=3$}: 
		\begin{itemize}
			\item $\ng_3$: $[X_3,X_2]=X_1$
		\end{itemize}
		\item {\it Case $n=4$}: 
		\begin{itemize}
			\item $\ng_4$: $[X_4,X_3]=X_2$, $[X_4,X_2]=X_1$
		\end{itemize}
		\item {\it Case $n=5$}: 
		\begin{itemize}
			\item $\ng_{5,1}$: $[X_5,X_4]=X_1$, $[X_3,X_2]=X_1$
			\item $\ng_{5,2}$: $[X_5,X_4]=X_2$, $[X_5,X_3]=X_1$
			\item $\ng_{5,3}$: $[X_5,X_4]=X_2$, $[X_5,X_2]=X_1$, $[X_4,X_3]=X_1$
			\item $\ng_{5,4}$: $[X_5,X_4]=X_3$, $[X_5,X_3]=X_2$, $[X_4,X_3]=X_1$
			\item $\ng_{5,5}$: $[X_5,X_4]=X_3$, $[X_5,X_3]=X_2$, $[X_5,X_2]=X_1$
			\item $\ng_{5,6}$: $[X_5,X_4]=X_3$, $[X_5,X_3]=X_2$, $[X_5,X_2]=X_1$, $[X_4,X_3]=X_1$
		\end{itemize}
	\end{enumerate}
	
With the above notation we can state the classification of nilpotent Lie algebras of dimension $\le 5$.
	\begin{proposition}\label{5D_classif}
		Every nilpotent Lie algebra of dimension $\le 5$ over $\RR$ is isomorphic to exactly one of the following Lie algebras: 
		\begin{itemize}
			\item Dimension 1: $\ag_1$
			\item Dimension 2: $\ag_2$  
			\item Dimension 3: $\ag_3$, $\ng_3$
			\item Dimension 4: $\ag_4$, $\ag_1\times\ng_3$, $\ng_4$
			\item Dimension 5: $\ag_5$, $\ag_2\times\ng_3$, $\ag_1\times\ng_4$, $\ng_{5,1}$, $\ng_{5,2}$, $\ng_{5,3}$, $\ng_{5,4}$, $\ng_{5,5}$, $\ng_{5,6}$
		\end{itemize}
	\end{proposition}
	
	\begin{proof}
		See \cite[Prop. 1]{Dix58}.
	\end{proof}

	

\begin{lemma}\label{L5}
	If $G$ is a nilpotent Lie group with $\dim G\le 5$ and $G$ is not isomorphic to the filiform group $F_5$, then $G$ is type~$\RRa_1$. 
\end{lemma}

\begin{proof}
	If all the coadjoint orbits of $G$ are flat, then $G$ is type~$\RRa_1$ by Proposition~\ref{P3}. 
	If not all the coadjoint orbits of $G$ are flat then, as a direct consequence of the classification given in Proposition~\ref{5D_classif}, $\gg$ is isomorphic to one of the Lie algebras $\ng_4=\fg_4$, $\ag_1\times\ng_4=\ag_1\times\fg_4$, $\ng_{5,4}$, $\ng_{5,6}$. 
	In any of these cases, we will prove that for any maximal element $S$ of $Q(\widehat G)$ one has $S\subseteq[\gg,\gg]^\perp$ and then, 
	by the maximality property of $S$ it follows that $S=[\gg,\gg]^\perp$. 
	
	We discuss the above four cases separately: 
	
	Case 1: $\gg=\ng_4=\fg_4$. 
	By \cite[Prop. 2]{Dix60} or \cite[Th. 5.4]{Fe62} we obtain open subsets 
	\begin{equation}\label{L5_proof_eq1}
	\emptyset=D_0\subset D_1\subset D_2\subset D_3=\widehat{F_4}
	\end{equation}
	where $D_1=D_1\setminus D_0\simeq \RR\times\RR^\times$ is dense in $\widehat{F_4}$, 
	$D_2\setminus D_1\simeq\RR^\times$, and $D_3\setminus D_2=[\fg_4,\fg_4]^\perp\simeq\RR^2$. 
	In addition, $D_1$ and $D_2\setminus D_1$ are not closed in $\widehat{F_4}$ and all the hypotheses of Lemma~\ref{L5b} are satisfied with $n=3$ and $X=\widehat{F_4}$. 
	It then follows by that lemma that for every maximal element $S$ of $Q(\widehat G)$ one has $S\subseteq \widehat{F_4}\setminus D_2=[\fg_4,\fg_4]^\perp$, and we are done. 
	
	Case 2: $\gg=\ag_1\times\ng_4=\ag_1\times\fg_4$. 
	Using the notation of Case~1, we obtain 
	$$\emptyset=D_0\subset \ag_1^*\times D_1\subset \ag_1^*\times D_2\subset \ag_1^*\times D_3=\widehat{G}$$
	and this family of open subsets of $\widehat{G}$ satisfies all the hypotheses of Lemma~\ref{L5b} since so does the family \eqref{L5_proof_eq1} from Case~1. 
	By that lemma we then obtain again that for any maximal element $S$ of $Q(\widehat G)$ one has $S\subseteq[\gg,\gg]^\perp$. 
	
	Case 3: $\gg=\ng_{5,4}$. 
	By \cite[Th. 5.6]{Fe62} or \cite[Th. 8.2]{LuRe15} we obtain the open subsets 
	$$\emptyset=D_0\subset D_1\subset D_2\subset D_3=\widehat{N_{5,4}}$$
	where $D_1=D_1\setminus D_0\simeq(0,\infty)\times\TT$ is dense in $\widehat{N_{5,4}}$, $D_2\setminus D_1\simeq\RR^\times$, 
	and $D_3\setminus D_2=[\ng_{5,4},\ng_{5,4}]^\perp\simeq\RR^2$. 
	In addition, $D_1$ and $D_2\setminus D_1$ are not closed in  $\widehat{N_{5,4}}$ and all the hypotheses of Lemma~\ref{L5b} are satisfied with $n=3$ and $X=\widehat{N_{5,4}}$. 
	By that lemma we then obtain again that for any maximal element $S$ of $Q(\widehat G)$ one has $S\subseteq[\gg,\gg]^\perp$. 
	
	Case 4: $\gg=\ng_{5,6}$. 
	By \cite[Prop. 7]{Dix60} we obtain the open subsets 
	$$\emptyset=D_0\subset D_1\subset D_2\subset D_3\subset D_4=\widehat{N_{5,6}}$$
	where $D_1=D_1\setminus D_0\simeq\RR^\times$ is dense in $\widehat{N_{5,6}}$, $D_2\setminus D_1\simeq\RR\times\RR^\times$, 
	$D_3\setminus D_2\simeq \RR^\times$, 
	and $D_4\setminus D_3=[\ng_{5,6},\ng_{5,6}]^\perp\simeq\RR^2$. 
	In addition, $D_1$, $D_2\setminus D_1$, and $D_3\setminus D_2$  are not closed in  $\widehat{N_{5,6}}$ and all the hypotheses of Lemma~\ref{L5b} are satisfied with $n=3$ and $X=\widehat{N_{5,4}}$. 
	By that lemma we then obtain again that for any maximal element $S$ of $Q(\widehat G)$ one has $S\subseteq[\gg,\gg]^\perp$, and this completes the proof.
\end{proof}

\begin{proposition}\label{P6}
	If $G$ is a nilpotent Lie group with $\dim G\le 5$, then $G$ is type~$\RRa$. 
\end{proposition}

\begin{proof}
	If $G$ is not isomorphic to the filiform group~$F_5$, then $G$ is type~$\RRa_1$ by Lemma~\ref{L5}, hence $G$ is type~$\RRa$ by Lemma~\ref{L2}. 
	
	If $G$ is isomorphic to~$F_5$, then $G$ is class~$\Tc$, hence $G$ is type~$\RRa$ by Proposition~\ref{P4}. 
	This completes the proof. 
\end{proof}

\begin{proposition}\label{P7}
	If $G_1$ and $G_2$ are nilpotent Lie groups satisfying the condition $\max\{\dim G_1,\dim G_2\}\le 5$, then one has 
	$$G_1\text{ is isomorphic to }G_2\iff
	\widehat{G_1}\text{ is homeomorphic to }\widehat{G_2}. $$
\end{proposition}

\begin{proof}
	Assume that $\widehat{G_1}$ is homeomorphic to $\widehat{G_2}$. 
	It follows by Lemma~\ref{27december2017-1555} that 
	\begin{equation}\label{P7_proof_eq1}
	\ind G_1=\ind G_2.
	\end{equation}
	Furthermore, if both $G_1$ and $G_2$ are isomorphic to the filiform group~$F_5$, then we are done. 
	Otherwise, if for instance $G_1$ is not isomorphic to $F_5$, then $G_1$ is type~$\RRa_1$ by Lemma~\ref{L5}. 
	It then follows by Lemma~\ref{L2} that 
	\begin{equation}\label{P7_proof_eq2}
	\RRa(C^*(G_1))=\RRa(C^*(G_2)).
	\end{equation}
	For any nilpotent Lie algebra $\gg$  and $G$ the corresponding 
	connected simply connected nilpotent Lie  group
	we have
	hence $\RRa(C^*(G))=\dim(\gg/[\gg,\gg])$ by~\eqref{RR_eq}. 
	
	On the other hand, $\gg_1$ and $\gg_2$ belong to the list of Lie algebras provided by Proposition~\ref{5D_classif}. 
	Using the detailed information that is available on the coarse partitions of the spaces of coadjoint orbits (see for instance \cite{Pe88}), one easily obtains the following values for the real rank and the index of the
	connected simply connected nilpotent Lie  groups corresponding to the Lie algebras in  the above mentioned list. 
	
	\begin{center}
		\begin{tabular}{ | c | c | c |}
			\hline
			$\gg$ 	& $\RRa(C^*(G))$ & $\ind(G)$  \\ \hline 
			$\ag_k$ ($1\le k\le 5$ )  	& $k$ &$ k$ \\ \hline
			$\ng_3(=\hg_3)$    	& 2 & 1 \\ \hline
			$\ag_1\times\ng_3(=\ag_1\times\hg_3)$ & 3 & 2 \\ \hline
			$\ng_4 (=\fg_4)$    & 2 & 2\\ \hline
			$\ag_2\times\ng_3 (=\ag_2\times\hg_3)$	& 4 &  3 \\ \hline
			$\ag_1\times\ng_4(=\ag_1\times\fg_4)$ & 3 & 3 \\ \hline
			$\ng_{5,1} (=\hg_5)$ & 4 & 1 \\ \hline
			$\ng_{5,2}$ & 3 & 3  \\ \hline
			$\ng_{5,3}$ & 3 & 1  \\ \hline
			$\ng_{5,4}$ & 2 & 3  \\ \hline
			$\ng_{5,5}(=\fg_5)$ & 2 & 3  \\ \hline
			$\ng_{5,6}$ & 2 & 1  \\ \hline
		\end{tabular}
	\end{center}
	
	If either $G_1$ or $G_2$ is a direct product of an abelian group with a Heisenberg group, then $G_1$ is isomorphic to $G_2$ by Theorem~\ref{2H}.

	Because of these observations, it suffices to compare to each other the unitary dual spaces of any two Lie groups from the above list from which one has removed the Heisenberg groups and the abelian groups. 
	Comparing the values of $\RRa$ and $\ind$ from the above table, 
	one easily checks that only in dimension~5 one encounters pairs of distinct Lie algebras with groups having the same real rank and index, 
	namely
	\begin{enumerate}
		\item $\ag_1\times\ng_4(=\ag_1\times\fg_4)$ and $\ng_{5,2}$,  
		\item $\ng_{5,4}$ and $\ng_{5,5}(=\fg_5)$.
	\end{enumerate}
	So, due to \eqref{P7_proof_eq1}--\eqref{P7_proof_eq2}, 
	it remains to compare to each other the unitary dual spaces of the above pairs, which can be done as follows: 
	\begin{enumerate}
		\item There is no homeomorphism between the unitary duals of the groups $G_1:=A_1\times F_4$ and $G_2:=N_{5,2}$. 
		
		Indeed, any homeomorphism between the unitary duals of the groups $G_1$ and $G_2$ should map to each other the spaces of characters of these groups, by Lemma~\ref{L5} and Lemma~\ref{L2}. 
		On the other hand, by \cite[Prop. 4]{Dix60}, the group $G_2$ has the property that the complement of the characters $\widehat{G_2}\setminus[\gg_2,\gg_2]^\perp$ is Hausdorff in its relative topology. 
		However, the relative topology of the complement of the characters $\widehat{G_1}\setminus[\gg_1,\gg_1]^\perp$ does not have the Hausdorff property as a direct consequence of  \cite[Prop. 2]{Dix60}.
		\item There is no homeomorphism between the unitary duals of the groups $N_{5,4}$ and $N_{5,5}(=F_5)$. 
		
		In fact, as noted in \cite[\S 5]{Lu90}, 
		there exists a properly convergent sequence in $\widehat{F_5}$ that has exactly three limit points. 
		
		On the other hand \cite[Thm. 5.6]{Fe62} (see also \cite[Thm. 8.2]{LuRe15}) shows that
		the set of limit points of any properly convergent sequence in $\widehat{N_{5,4}}$ contains either one point, or two points, or infinitely many points, as soon as we proved the following assertion:
		
		Assume that $\lim\limits_{n\to\infty}s_n=0$ in $\RR$, 
		$\lim\limits_{n\to\infty}w_n=0$ in $(0,\infty)$, 
		and $\{\theta_n\}_{n\ge 1}$ is a sequence of real numbers. 
		Then the set 
		\begin{equation}\label{ineq}
		\Bigl\{(u_1,u_2)\in\RR^2\mid
		\liminf\limits_{n\to\infty}
		\Bigl (\frac{s_n}{w_n}-u_1\sin\theta_n-u_2\cos\theta_n \Bigr )\le0
		\Bigr\}
		\end{equation}
		is either empty or infinite.

		To prove this, denote by $E\subseteq\RR^2$ the set in \eqref{ineq}, and assume that $E\ne \emptyset$, 
		so that we can select a point $(u_1,u_2)\in E$. 
		
		Restricting to a suitable subsequence, we may assume that 
		all the real numbers in the sequence $\{\sin\theta_n\}_{n\ge 1}$ have the same sign $\epsilon_1\in\{\pm1\}$, and likewise all the real numbers in the sequence  $\{\cos\theta_n\}_{n\ge 1}$ 
		have the same sign $\epsilon_2\in\{\pm1\}$. 
		
		Let $u_j'\in\RR$ arbitrary with $0\le(u_j'-u_j)\varepsilon_j$ 
		for $j=1,2$. 
		Then for every $n\ge 1$ one has 
		$0\le (u_1'-u_1)\varepsilon_1\vert\sin\theta_n\vert 
		= (u_1'-u_1)\sin\theta_n$ 
		and similarly $0\le  (u_2'-u_2)\cos\theta_n$. 
		We then obtain 
		$$(\forall n\ge 1)\quad 
		\frac{s_n}{w_n}-u_1'\sin\theta_n-u_2'\cos\theta_n
		\le \frac{s_n}{w_n}-u_1\sin\theta_n-u_2\cos\theta_n$$
		and, since $(u_1,u_2)\in E$, it then follows that $(u_1',u_2')\in E$. 
		As the set of points $(u_1',u_2')\in\RR^2$ with 
		$0\le(u_j'-u_j)\varepsilon_j$ 
		for $j=1,2$ is infinite, $E$ is infinite.

	\end{enumerate}
	This completes the proof. 
\end{proof}

\begin{remark}
	\normalfont
	Neither of Propositions \ref{P6} and \ref{P7} is stronger than the other one. 
\end{remark}

We now can take another step towards the proof of Theorem~\ref{5D}. 

\begin{proposition}\label{6D-MD}
	Let $G_1$ be a Lie group of class~$\Tc$ with $\dim G_1\le 5$. 
	If $G_2$ is a nilpotent Lie group for which $C^*(G_1)$ is Morita equivalent to $C^*(G_2)$, then $G_1$ is isomorphic to $G_2$. 
\end{proposition}

The proof of this proposition requires several lemmas.

\begin{lemma}\label{MD-indecomp}
	If $\gg$ is a nilpotent Lie algebra of class~$\Tc$ and there exist Lie algebras $\gg_1$ and $\gg_2$ with $\gg=\gg_1\times\gg_2$, then at least one of the Lie algebras $\gg_1$ and $\gg_2$ is abelian. 
\end{lemma}

\begin{proof}
	Assuming for $j=1,2$ that $\gg_j$ is not abelian, it follows that there exists $\xi_j\in\gg_j^*$ with $\{0\}\subsetneqq\gg_j(\xi_j)\subsetneqq\gg_j$. 
	For $\xi:=(\xi_1,\xi_2)\in\gg_1^*\times\gg_2^*=\gg^*$ 
	one has $\gg(\xi)=\gg_1(\xi_1)\times\gg_2(\xi_2)$. 
	Similarly, for 
	$\eta:=(0,\xi_2)\in\gg_1^*\times\gg_2^*=\gg^*$ 
	one has $\gg(\eta)=\gg_1\times\gg_2(\xi_2)$, 
	hence $\gg(\eta)\subsetneqq\gg(\xi)\subsetneqq\gg$, and this shows that $\gg$ is not of class~$\Tc$, which is a contradiction with the hypothesis. 
\end{proof}

\begin{lemma}\label{MD-red}
	A Lie algebra $\gg$ is of class~$\Tc$ if and only if there exist an 
	integer $k\ge 0$ and an 
	indecomposable Lie algebra $\gg_0$ of class~$\Tc$ with $\gg=\ag_k\times\gg_0$. 
\end{lemma}

\begin{proof}
	If $\gg_0$ is of class~$\Tc$, then it is easily checked that $\ag_k\times\gg_0$ is of class~$\Tc$. 
	
	We prove the converse assertion by induction on $\dim\gg$. 
	If $\dim\gg=1$, then $\gg=\ag_1$, and we are done. 
	
	For the induction step, if $\gg$ is indecomposable, then we may set $k:=0$ and $\gg_0:=\gg$. 
	If $\gg$ is not indecomposable, then there exist Lie algebras $\gg_1$ and $\gg_2$ with  $\gg=\gg_1\times\gg_2$ and $\dim\gg_j\ge 1$ for $j=1,2$. 
	Since $\gg$ is of class~$\Tc$, it follows by Lemma~\ref{MD-indecomp} that one of the Lie algebras $\gg_1$ and $\gg_2$ is abelian. 
	We may assume that there exists an integer $k_1\ge 1$ with $\gg_1=\ag_{k_1}$. 
	Thus $\gg=\ag_{k_1}\times\gg_2$ with $\dim\gg_2<\dim\gg$. 
	Since $\gg$ is of class~$\Tc$, it is straightforward to check that 
	$\gg_2$ is of class~$\Tc$ and then, by the induction hypothesis, 
	there exist an integer $k_2\ge 0$ an an indecomposable Lie algebra $\gg_0$ of class~$\Tc$ with $\gg_2=\ag_{k_2}\times\gg_0$. 
	Thus $\ag=\ag_{k_1}\times\ag_{k_2}\times\gg_0=\ag_{k_1+k_2}\times\gg_0$, and this completes the induction step. 
\end{proof}

\begin{lemma}\label{29March2018}
	If $G_1$ and $G_2$ are simply connected solvable Lie groups for which $C^*(G_1)$ and $C^*(G_2)$ are Morita equivalent, then $\dim G_1-\dim G_2$ is an even integer. 
\end{lemma}

\begin{proof}
	It follows by \cite[Sect. V, Cor. 7]{Co81} 
	that for any simply connected solvable Lie group $G$ there is a group isomorphism $K_0(C^*(G))\simeq K_{j_G}(\CC)$, where $j_G=0$ if $\dim G$ is an even integer, and $j_G=1$ if $\dim G$ is an odd integer. 
	Moreover, one has a group isomorphism $K_0(C^*(G))\simeq K_0(C^*(G)\otimes\Kc)$  
	by \cite[Cor. 6.2.11]{WO93}.  
	Thus the hypothesis implies $K_{j_{G_1}}(\CC)\simeq K_{j_{G_2}}(\CC)$. 
	
	On the other hand, 
	 it is well known that $K_0(\CC)=\ZZ$ and $K_1(\CC)=\{0\}$ (cf. \cite[\S 6.5]{WO93}), hence the existence of a group isomorphism $K_{j_{G_1}}(\CC)\simeq K_{j_{G_2}}(\CC)$ implies $j_{G_1}=j_{G_2}$, 
	and now the assertion follows at once. 
\end{proof}

\begin{proof}[Proof of Proposition~\ref{6D-MD}]
	By a simple analysis using Proposition~\ref{5D_classif} we see that every Lie algebra of class~$\Tc$ having dimension $n\le 5$ is isomorphic to precisely one of the following Lie algebras: 
	\begin{enumerate}[{\rm(a)}]
		\item {\it Case $n=1$}: $\ag_1$
		\item {\it Case $n=1$}: $\ag_2$
		\item {\it Case $n=3$}: $\ag_3$, $\ng_3$
		\item {\it Case $n=4$}: $\ag_4$, $\ag_1\times\ng_3$, $\ng_4$
		\item {\it Case $n=5$}: $\ag_5$, $\ag_2\times\ng_3$, $\ag_1\times\ng_4$, $\ng_{5,1}$, $\ng_{5,2}$, $\ng_{5,4}$, $\ng_{5,5}$ 
	\end{enumerate}
	We may assume without loss of generality that $G_1$ is nonabelian. 
	One has 
	\begin{equation}\label{6D-MD_proof_eq0}
	\ind G_1=\ind G_2\text{ and }\dim[\gg_1,\gg_1]^\perp=\dim[\gg_2,\gg_2]^\perp
	\end{equation}
	by Remark~\ref{27december2017-1600} and 
	Proposition~\ref{P4}, 
	respectively. 
	On the other hand, it follows by Proposition~\ref{P4} that $G_2$ is of class~$\Tc$ and then, 
	by Lemma~\ref{MD2}, we obtain 
	\begin{equation}\label{6D-MD_proof_eq1}
	\dim\gg_2\le \ind G_1+\dim[\gg_1,\gg_1]^\perp.
	\end{equation}

	Therefore we  need to discuss the cases below. 
	
	$\bullet$ $\dim\gg_1=3$. 
	Then $\gg_1=\ng_3$, 
	hence $\ind G_1=1$ and 
	$\dim[\gg_1,\gg_1]^\perp=2$.

	$\bullet$ $\dim\gg_1=4$. 
	If $\gg_1=\ag_1\times\ng_3$, then $\ind G_1=2$ and 
	$\dim[\gg_1,\gg_1]^\perp=3$. 
	If $\gg_1=\ng_4$, then $\ind G_1=2$ and 
	$\dim[\gg_1,\gg_1]^\perp=2$. 
	
	Thus, if $\dim G_1\le 4$, then we obtain $\dim G_2\le 5$ by \eqref{6D-MD_proof_eq1}, hence $G_1$ is isomorphic to $G_2$ by 
	Proposition~\ref{P7}. 
	
	$\bullet$ $\dim\gg_1=5$. 
	If $\gg_1=\ag_2\times\ng_3$, then Theorem~\ref{2H} is applicable. 
 	If $\gg_1$ is one of the Lie algebras 
	$\ag_1\times\ng_4$, $\ng_{5,1}$, $\ng_{5,2}$, $\ng_{5,4}$, or $\ng_{5,5}$, 
	then an inspection of the table from the proof of Proposition~\ref{P7}. 
	shows that 
	$\ind G_1+\dim[\gg_1,\gg_1]^\perp\le 6$. 
	
	Thus, if $\dim G_1=5$, then $\dim G_2\le 6$ by \eqref{6D-MD_proof_eq1}. 
	Moreover, Lemma~\ref{29March2018} shows that $\dim G_2\ne 6$, hence $\dim G_2\le 5$, and then 
	$G_1$ is isomorphic to $G_2$ by 
	Corollary~\ref{5D}. 
	\end{proof}

\subsection{Proof of Theorem~\ref{5D}}

\begin{lemma}\label{30March2018}
	Let $G_1$ and $G_2$ be nilpotent Lie groups for which $C^*(G_1)$ is Morita equivalent to $C^*(G_2)$, 
	and denote by $Z_j$ the center of $G_j$ for $j=1,2$. 
	If $\ind G_1=1$, then $\ind G_2=\dim Z_2=\dim Z_1=1$ and $C^*(G_1/Z_1)$ is Morita equivalent to $C^*(G_2/Z_2)$. 
\end{lemma}

\begin{proof}
	One has $\ind G_2=\ind G_1=1$	
	by Remark~\ref{27december2017-1600}. 
	Therefore, by Lemma~\ref{27december2017-1555}, 
	one has $\dim Z_2=\dim Z_1=1$ 
	and there exists $\xi_j\in\gg_j^*$ with $\gg_j(\xi_j)=\zg_2$ and moreover the mapping $\RR^\times\to\gg_j^*/G_j$, $t\mapsto\Ad^*_{G_j}(G_j)(t\xi_j)$, is a homeomorphism of $\RR^\times$ onto an open dense subset $D_j$ of $\gg_j^*/G_j$. 
	The set $D_j$ is the set of all coadjoint orbits of $G_j$ having maximal dimension, whose union is 
	$\{\eta\in\gg_j^*\mid \zg\not\subseteq\Ker\eta\}$. 
	Let $\Jc_j\subseteq C^*(G_j)$ be the closed two-sided ideal with $\widehat{\Jc_j}=D_j$. 
	Then there is  a short exact sequence of $C^*$-algebras
	$$0\to\Jc_j\to C^*(G_j)\to C^*(G_j/Z_j)\to 0.$$
	Moreover, $\Jc_j$ is the largest bounded-trace ideal of $C^*(G_j)$ (see \cite[Sect.~2]{AKLSS01}).
	 
	It then follows by \cite[Cor. 9]{aHRaWi07} that for any fixed imprimitivity $C^*(G_1)$-$C^*(G_2)$-bimodule 
	its corresponding Rieffel correspondence carries $\Jc_1$ to $\Jc_2$. 
	Now, using \cite[Prop. 3.25]{RaWi98}, we obtain that the quotients 
	$C^*(G_1)/\Jc_1$ and $C^*(G_2)/\Jc_2$ are Morita equivalent. 
	Then, taking into account the above short exact sequences, 
	the $C^*$-algebras $C^*(G_1/Z_1)$ and $C^*(G_2/Z_2)$ are Morita equivalent. 
	This completes the proof. 
\end{proof}

\begin{proof}[Proof of Theorem~\ref{5D}] 
	Let $G_1$ be a nilpotent Lie group of dimension $\le 5$ and $G_2$ an exponential Lie group such that $C^*(G_1)$ is Morita equivalent to 
	$C^*(G_2)$. 
	We must prove that the Lie groups $G_1$ and $G_2$ are isomorphic. 
	
	First, Lemma~\ref{GCR} implies that $G_2$ must be nilpotent. 
	
	Then, if $G_1$ is of class~$\Tc$, 
	the assertion follows by Proposition~\ref{6D-MD}. 
	
	Now let us assume that $G_1$ is not of class~$\Tc$. 
	It follows by  Proposition~\ref{5D_classif} and the list in the proof of Proposition~\ref{6D-MD}
	that the only $5$-dimensional nilpotent Lie algebras which are not of class~$\Tc$ are $\ng_{5,3}$ and $\ng_{5,6}$. 
	Let us denote the center of $\gg_j$ by $\zg_j$ for $j=1,2$. 
	
	If either $\gg_1=\ng_{5,3}$ or  $\gg_1=\ng_{5,6}$, then $\ind G_1=1$, 
	hence by Lemma~\ref{30March2018} we obtain that $C^*(G_1/Z_1)$ is Morita equivalent to $C^*(G_2/Z_2)$ and $\dim Z_2=1$. 
	Here $G_1/Z_1$ is isomorphic either to $A_1\times H_3$ (if $\gg_1=\ng_{5,3}$)
	or to $N_4$ (if $\gg_1=\ng_{5,6}$). 
	Both Lie groups $A_1\times H_3$ and $N_4$ are 4-dimensional and are of class~$\Tc$, hence by Proposition~\ref{6D-MD} we obtain that $G_1/Z_1$ is isomorphic to $G_2/Z_2$. 
	In particular $\dim G_2=5$, and then $G_1$ is isomorphic to $G_2$ by 
	Proposition~\ref{P7}. 
	This completes the proof. 
\end{proof}

\section{Other examples }\label{other_ex}

So far we established that the class of stably $C^*$-rigid groups contains all nilpotent Lie groups of dimension $\le 5$ and direct products of Heisenberg groups with abelian Lie groups.
We show in this section that the above class also contains the filiform Lie groups and some $6$-dimensional nilpotent Lie groups.

	\subsection{Filiform Lie groups}
	Lie algebras of the filiform Lie groups are defined as follows:
	For $n:=\dim\gg\ge 3$,  the nilpotent Lie algebra $\fg_n$ has a basis $X_1,\dots,X_n$ with the commutation relations 
	$$[X_n,X_j]=X_{j-1} \quad \text{for} \quad j=1,\dots,n-1, $$
	where $X_0:=0$, 
	and $[X_k,X_j]=0$ if $1\le j\le k\le n-1$.

	\begin{proposition}\label{rigid_fili}
		For any $m,n\ge 3$,  $\widehat{F_n}$ and $\widehat{F_m}$ are homeomorphic if and only if $n=m$.  
			\end{proposition}
	
	\begin{proof}
		For every $n\ge 3$ one has $\ind F_n=n-2$, and then the assertion follows by Remark~\ref{27december2017-1600}. 
	\end{proof}
	
	\begin{remark}
		\normalfont
		The stronger hypothesis that $C^*(F_m)$ is Morita equivalent to
	$C^*(F_n)$ implies that $m=n$, by \cite[Thms.~4.1 and 4.8]{AKLSS01}, \cite[Prop.~2.2]{ArSp96} and 
\cite[Thm.~10]{aHRaWi07}.
		\end{remark}
	
	We introduce here the $6$-dimensional free 2-step nilpotent Lie algebra, denoted 
	$\ng_{6,15}$ in \cite{Pe88},  that is, the Lie algebra defined by a basis $X_1,X_2,X_3,X_4,X_5,X_6$ satisfying the commutation relations $$[X_6,X_5]=X_3,\ [X_6,X_4]=X_1,\ [X_5,X_4]=X_2.$$
This will be needed in the proof of the next theorem, and also treated in Subsection~\ref{free}.

\begin{theorem}\label{fili_th}
	The filiform Lie group $F_n$ is stably $C^*$-rigid for every $n\ge 3$.
\end{theorem}
We first prove the following lemma. 
\begin{lemma}\label{fili_L}
	If $\gg$ is a nilpotent Lie algebra of class~$\Tc$ with $\RRa(C^*(G))\le
	3$, then $\ind\gg=\dim\gg-2$.
\end{lemma}
\begin{proof}
	Since $\gg$ is class~$\Tc$, it follows by Lemma~\ref{MD1} that there
	exists $\xi\in\gg^*$ with
	$$[\gg,\gg]\subseteq\gg(\xi)\subsetneqq\gg.$$
	Here $2\le \dim(\gg/[\gg,\gg])=\RRa(C^*(G))\le 3$ by hypothesis,
	hence $2\le \dim(\gg/\gg(\xi))\le 3$.
	On the other hand $\dim(\gg/\gg(\xi))$ is an even integer,
	hence $\dim(\gg/\gg(\xi))=2$.
	Since $\gg$ is class~$\Tc$, we then obtain $\ind\gg=\dim\gg(\xi)=\dim\gg-2$.
\end{proof}
\begin{proof}[Proof of Theorem~\ref{fili_th}]
	We must prove that if $G$ is an exponential Lie group for which
	$C^*(G)$ is Morita equivalent to $C^*(F_n)$,
	then $G$ is isomorphic to the  Lie group~
	$F_n$.
	\par
	It follows by Lemma~\ref{GCR} that $G$ is a nilpotent Lie group.
	Moreover,
	$\RRa(C^*(G))=\RRa(C^*(F_n))=2$  
	and $G$ is class~$\Tc$ by
	Proposition~\ref{P4}.
	Therefore we may use Lemma~\ref{fili_L} to obtain
	$\ind\gg=\dim\gg-2$, that is, all the non-trivial coadjoint orbits of $\gg$ have
	dimension~$ 2$.
	Since $\gg$ is nilpotent, it then follows by \cite[Thm.]{ACL95} that
	one of the following cases may occur:
	\par
	Case 1: There exists a hyperplane abelian ideal of~$\gg$.
	\par
	Case 2: There exists an integer $k\ge 0$ with $\gg=\ag_k\times\ng_{6,15}$.
	\par
	Case 3: There exists an integer $k\ge 0$ with $\gg=\ag_k\times\ng_{5,4}$. \\
	In Case 2 one has $2=\RRa(C^*(G))=\RRa(C^*(A_k\times N_{6,15}))
	=k+3$, hence $k=-1$, which is impossible.
	In Case 3, one has  similarly that $2= \RRa(C^*(G))=\RRa(C^*(A_k\times N_{5,4}))
	=k+2$, hence $k=0$.
	That is, $\gg=\ng_{5,4}$, but this is impossible since
	Theorem~\ref{5D} shows that $C^*(N_{5,4})$ is not Morita equivalent to
	$C^*(F_n)$.
	\par
	It follows by the above discussion that only Case 1 can occur, hence
	there exists
	an abelian ideal $\ag\trianglelefteq\gg$ with $\dim(\gg/\ag)=1$.
	Let us denote $m:=\dim\gg$ and select any $X\in\gg\setminus\ag$.
	Then
	$[\gg,\gg]=[X,\ag]$, since $[\ag,\ag]=\{0\}$.
	Since $\dim[\gg,\gg]=\dim\gg-\RRa(C^*(G))=m-2$, it then follows that
	the operator $D:=(\ad_{\gg} X)\vert_{\ag}\colon\ag\to\ag$ is
	nilpotent and its range has codimension~$1$ in $\ag$.
	This implies that the Jordan decomposition of $D$ consists of exactly
	one Jordan cell, which further implies that~$\gg$ is isomorphic to the
	filiform Lie algebra~$\fg_m$.
	Since $C^*(G)$ and $C^*(F_n)$ are Morita equivalent, it then follows
	by Proposition~\ref{rigid_fili} that $m=n$, hence $G$ is isomorphic
	to~$F_n$, and this completes the proof.
\end{proof}

	\subsection{The groups $H_{m,n}$}
		For any $m,n\ge1$,  $H_{m,n}$ are the nilpotent Lie groups with Lie algebras $\hg_{m,n}$  
	with a basis $\{X_1,\dots,X_m\}\cup\{Y_0,\dots,Y_n\}$ and the bracket given by  
	$$[X_i,Y_j]=Y_{i+j}$$
	for all $i\in\{1,\dots,m\}$ and $j\in\{0,\dots,n\}$ with $i+j\le n$. 
	   We recall that $\dim H_{m,n}=m+n+1$ for all $m\ge n\ge 1$, and all the coadjoint orbits of $H_{m, n}$ are flat.
	   (See \cite[Subsect.~6.2]{BBL17} and the references therein.)
	We also define $H_{m,0}:=A_{m+1}$, the $(m+1)$-dimensional abelian Lie group.

	For any $C^*$-algebra $\Ac$ we will denote by $\Ic(\Ac)$ the set of all closed 2-sided ideals of~$\Ac$.

		\begin{proposition}\label{rigid_flat}
		If $m_1\ge n_1\ge 1$ and $m_2\ge n_2\ge 1$, then $C^*(H_{m_1,n_1})$ is Morita equivalent to $C^*(H_{m_2,n_2})$ if and only if $m_1=m_2$ and $n_1=n_2$. 
	\end{proposition}
	
	\begin{proof}
		Denote $\Ac_k:=C^*(H_{m_k,n_k})$ for $k=1,2$. 
		Assume that $\Ac_1$ is Morita equivalent to $\Ac_2$, and let $X$ be an $\Ac_1-\Ac_2$-imprimitivity bimodule with  Rieffel correspondence $X$-$\Ind\colon \Ic(\Ac_2)\to \Ic(\Ac_1)$. 
		If $J_k\in \Ic(\Ac_k)$ is the largest bounded-trace ideal of $\Ac_k$ for $k=1,2$ then, 
		using \cite[Thm. 2.8 and Cor. 2.9]{ArSoSp97}, 
		one obtains the canonical homeomorphism $\widehat{J_k}\simeq\Gamma_k$, 
		where $\Gamma_k\subseteq\widehat{\Ac_k}$ 
		is the open subset corresponding to the coadjoint orbits of $H_{m_k,n_k}$ having maximal dimension.  
		The short exact sequence \cite[Eq. (6.3)]{BBL17} then takes on the form 
		$$0\to J_k\to \Ac_k\to C^*(H_{m_k,n_k-1})\to 0.$$
		On the other hand, since $J_k$ is the largest bounded-trace ideal of $\Ac_k$,  one has 
		$X$-$\Ind(J_2)=J_1$ by \cite[Cor. 9]{aHRaWi07}, 
		and it further follows by \cite[Prop. 3.25]{RaWi98} that $\Ac_1/J_1$ is Morita equivalent to $\Ac_2/J_2$ for $k=1,2$. 
		The above short exact sequence then shows that 
		$C^*(H_{m_1,n_1-1})$ is Morita equivalent to $C^*(H_{m_2,n_2-1})$. 
		
		Now let us assume that $n_1\le n_2$. 
		Iterating the above reasoning, we obtain that 
		$C^*(H_{m_1,0})$ is Morita equivalent to $C^*(H_{m_2,n_2-n_1})$. 
		Since $H_{m_1,0}$ is the abelian $(m_1+1)$-dimensional Lie group, 
		it then follows by 
		Lemma~\ref{abelian} 
		that $H_{m_2,n_2-n_1}$ is abelian and in fact is isomorphic to $H_{m_1,0}$, and then $n_2-n_1=0$ and $m_2=m_1$, 
		which concludes the proof.
	\end{proof}

\subsection{The 6-dimensional free 2-step nilpotent Lie group}\label{free}
We now turn our attention towards the  $6$-dimensional free 2-step nilpotent Lie algebra defined just before Theorem~\ref{fili_th}.

\begin{theorem}\label{N6-15}
The nilpotent Lie group $N_{6,15}$ is stably $C^*$-rigid. 
\end{theorem}


\begin{proof}
	Let $G$ be an exponential Lie group for which $C^*(G)$ is Morita equivalent to $C^*(N_{6,15})$. 
	We must prove that $G$ is isomorphic to $N_{6,15}$.
	
	It is well known that the coadjoint orbits of $N_{6,15}$ have dimensions $\le 2$ and then we directly obtain 
	$\ind N_{6,15}=4$. 
	Thus $\ind G=\ind N_{6,15}=4$	
	by Remark~\ref{27december2017-1600}. 
	Also, since $N_{6,15}$ is class~$\Tc$, 
	it follows by 
		Proposition~\ref{P4}
	that $G$ is class~$\Tc$ and 
	 $$\dim [\gg, \gg]^\perp=\dim[\ng_{6,15},\ng_{6,15}]^\perp=3. $$ 
 On the other hand, 
	since $G$ is class~$\Tc$, 
	we obtain $\dim\gg\le 4+3=7$ by Lemma~\ref{MD2}. 
	Now, by Lemma~\ref{29March2018}, it follows that $\dim\gg\in\{2,4,6\}$. 
	We discuss these cases separately below. 
	
	$\bullet$ Case $\dim\gg=2$. 
	Then $\gg$ is abelian, hence 
	$\dim [\gg, \gg]^\perp=2$, which is a contradiction with the equality 
	$\dim [\gg, \gg]^\perp=3$ established above. 
	
	$\bullet$ Case $\dim\gg=4$. 
	If $\gg$ is abelian, then $\dim[\gg, \gg]^\perp=4$, which is a contradiction as above. 
	If $\gg$ is not abelian, then $\gg=\hg_3\times\ag_1$ and then, by Theorem~\ref{2H}, we obtain that the group $N_{6,15}$ is isomorphic to $H_3\times A_1$, which is again a contradiction. 
	
	$\bullet$ Case $\dim\gg=6$. 
	There are two possible subcases. 
	
	Subcase 1: There exist Lie algebras $\gg_1$ and $\gg_2$ 
	with $\gg=\gg_1\times\gg_2$ and $\dim\gg_j\ge 1$ for $j=1,2$. 
	We may assume $\dim\gg_1\le\dim\gg_2$ without loss of generality. 
	Since $\dim\gg_1+\dim\gg_2=\dim\gg=6$, we may have either $\dim\gg_1=1$, or $\dim\gg_1=2$, or $\dim\gg_1=3$. 
	
	If neither $\gg_1$ nor $\gg_2$ is abelian, then $\gg$ is not of class $\Tc$, which is a contradiction with what we already established above. 

	If one of the Lie algebras $\gg_1$ and $\gg_2$ is abelian, 
	then we have either $\dim\gg_1\le 2$, or $\dim\gg_1=3$ and $\gg_1$ is abelian, 
	hence $\gg_1=\ag_k$ with $k\in\{1,2,3\}$. 
	Then 
	$$
	\begin{aligned}
	3= & \dim[\gg, \gg]^\perp=\RRa(C^*(G))=
	\RRa(C^*(A_k\times G_2))=k+\RRa(C^*(G_2))\\
	= & k+\dim[\gg_2,\gg_2]^\perp,
	\end{aligned}
	$$
	hence $\dim[\gg_2,\gg_2]^\perp=3-k$. 
	Since $\gg_2$ is nilpotent, it then follows that $k\le 1$, hence $k=1$, 
	and then $\RRa(C^*(G_2))=2$ and $\gg=\ag_1\times\gg_2$. 
	On the other hand, 
	$4=\ind \gg=\ind(\ag_1\times\gg_2)=1+\ind\gg_2$, hence $\ind\gg_2=3$. 
	An inspection of the table from the proof of Proposition~\ref{P7} 
	shows that the only  nilpotent Lie algebras $\gg_2$ with $\dim\gg_2=5$, $\RRa(C^*(G_2))=1$, and $\ind\gg_2=3$ are $\ng_{5,4}$ and $\ng_{5, 5}$.
	But this is impossible: The group $N_{6, 15}$ is two-step nilpotent and of class $\Tc$, therefore the relative topology of
	$\widehat{N_{6,15}}\setminus [\ng_{6, 15}, \ng_{6, 15}]^\perp$ is Hausdorff, while this property is not shared by the complement of characters of any of the groups $A_1\times N_{5, 4}$ and 
	$A_1\times N_{5, 5}$, since suitable quotients of these groups are isomorphic to the filiform group $F_{4}= N_{4}$.

	Subcase 2: The Lie algebra $\gg$ is indecomposable, that is, there exist no Lie algebras $\gg_1$ and $\gg_2$ 
	with $\gg=\gg_1\times\gg_2$ and $\dim\gg_j\ge 1$ for $j=1,2$. 
	Since $\dim\gg=6$, it then follows that $\gg$ is one of the 24 Lie algebras labeled as N6N1, N6N2,\dots N6N24 in \cite{Pe88}. 
	Since we already established that $\ind\gg=4$ and $\dim[\gg, \gg]^\perp=3$, it then easily follows by a direct inspection that either $\gg=\ng_{6,15}$ or $\gg=\ng_{6,18}$, where $\ng_{6,18}$ is the Lie algebra defined by a basis $X_1,X_2,X_3,X_4,X_5,X_6$ satisfying the commutation relations $$[X_6,X_5]=X_3,\ [X_6,X_4]=X_2,\ [X_6,X_3]=X_1.$$
	To complete the proof, we must show that,  assuming $\gg=\ng_{6,18}$, one obtains a contradiction. 
	In fact, if we define $\hg:=\RR X_2 +\RR X_4$, then $\hg$ is an ideal of $\ng_{6,18}$ for which the quotient $\ng_{6,18}/\hg$ is isomorphic to the 4-dimensional filiform Lie algebra~$\ng_4$  that occurs in Proposition~\ref{5D_classif}. 
	Therefore the unitary dual $\widehat{N_4}$ is homeomorphic to a closed subset of $\widehat{N_{6,18}}$ via a homeomorphism that takes the characters $[\ng_4, \ng_4]^\perp$ of $N_4$ to characters 
	$[\ng_{6,18}, \ng_{6,18}]^\perp$ of 
	$N_{6,18}$. 
	(See for instance \cite[Lemme 3]{Dix60}.) 
	The relative topology in $\widehat{N_4}$ of 
	$\widehat{N_4} \setminus [\ng_4, \ng_4]^\perp$ 
	is not Hausdorff 
	by \cite[Prop. 2 3]{Dix60}. 
	Hence  the relative topology in $\widehat{N_{6,18}}$ 
	of $\widehat{N_{6, 18}} \setminus [\ng_{6,18}, \ng_{6,18}]^\perp$  is not Hausdorff. 
	
	On the other hand, since $N_{6,15}$ is a 2-step nilpotent Lie groups, it follows that all its coadjoint orbits are flat. 
	Since the coadjoint orbits of $N_{6,15}$ have dimensions $\le 2$, it then follows by \cite[Lemma 6.8(1)]{BBL17} that the relative topology in $\widehat{N_{6,15}}$ 
	of  $\widehat{N_{6, 15}} \setminus [\ng_{6,15}, \ng_{6,15}]^\perp$is Hausdorff. 
	Taking into account the above remarks on $N_{6,18}$, it then follows that $C^*(N_{6,15})$ and $C^*(N_{6,18})$ are not Morita equivalent. 
	Indeed, if these two $C^*$-algebras were Morita equivalent then, 
	by Proposition~\ref{P4}
	one obtains a homeomorphism from $\widehat{N_{6,15}}$ onto $\widehat{N_{6,18}}$ that takes the characters of $N_{6,15}$ onto the characters of $N_{6,18}$. 
	This homeomorphism will then map the complement of the characters of $N_{6,15}$ homeomorphically onto the the complement of the characters of $N_{6,18}$, which is a contradiction with the fact that one of these spaces is Hausdorff while the other is not, as established above. 
	
	The assumption $\gg=\ng_{6,18}$ thus leads to a contradiction, and then there remains the fact that $\gg=\ng_{6,15}$, which completes the proof. 
\end{proof}

	\begin{remark}\label{N6-18}
		\normalfont
	A by-product of the proof of Theorem~\ref{N6-15} is that the Lie group 
	$N_{6, 18}$ is stably $C^*$-rigid as well. 
	\end{remark}


\subsection*{Acknowledgment} 
We wish to thank the Referee for several remarks that helped us improve our manuscript.


\end{document}